\documentclass[reqno,10pt,a4paper]{amsart}

\numberwithin{equation}{section}
\allowdisplaybreaks 

\usepackage{amssymb,eucal,mathrsfs,setspace}
\usepackage[noadjust]{cite}
\usepackage[inner=28mm, outer=28mm, top=27mm, bottom=28mm, head=10mm, foot=10mm]{geometry}

\usepackage{array}
\newcolumntype{C}{>{$}c<{$}} 

\usepackage[largesc,theoremfont]{newtxtext}      
\usepackage[libertine,cmbraces,varbb]{newtxmath} 
\begingroup
  \makeatletter
  \@for\theoremstyle:=definition,remark,plain\do{%
    \expandafter\g@addto@macro\csname th@\theoremstyle\endcsname{%
      \addtolength\thm@preskip{.5\baselineskip plus .2\baselineskip minus .2\baselineskip}
      \addtolength\thm@postskip{.5\baselineskip plus .2\baselineskip minus .2\baselineskip}
    }%
  }
\endgroup

\usepackage{enumitem}
\setitemize{leftmargin=*}   
\setenumerate{leftmargin=*, 
  label=\textup{(\alph*)},  
	ref=\textup{\alph*}}      

\usepackage{mathtools} 
\usepackage{graphicx}

\usepackage{tikz}
\usetikzlibrary{calc,backgrounds}
\tikzstyle{bwt}=[circle, fill=black, inner sep=2pt, outer sep=0pt, minimum size=5pt]  
\tikzstyle{dwt}=[circle, draw=black, fill=white, inner sep=2pt, outer sep=0pt, minimum size=5pt] 
\tikzstyle{rwt}=[circle, draw=black, fill=gray, inner sep=2pt, outer sep=0pt, minimum size=5pt] 

\usepackage[colorlinks=true,citecolor=red,linkcolor=blue]{hyperref} 

\usepackage[capitalise,noabbrev]{cleveref} 

\newcommand{\lam}{\lambda}
\newcommand{\Lam}{\Lambda}

\renewcommand{\ge}{\geq}
\renewcommand{\le}{\leq}

\renewcommand{\cong}{\simeq}

\DeclareMathOperator{\vspn}{span}

\DeclareMathOperator{\tr}{tr} 
\DeclareMathOperator{\Supp}{supp}
\DeclareMathOperator{\EssSupp}{ess-supp}
\DeclareMathOperator{\rad}{rad}

\newcommand{\kk}{\mathsf{k}}   

\newcommand{\wun}{\vvmathbb{1}}  

\DeclarePairedDelimiter{\brac}{\lparen}{\rparen}   
\DeclarePairedDelimiter{\sqbrac}{\lbrack}{\rbrack} 
\DeclarePairedDelimiter{\set}{\lbrace}{\rbrace}
\newcommand{\st}{\mspace{5mu} {:} \mspace{5mu}}    
\DeclarePairedDelimiter{\abs}{\lvert}{\rvert}

\DeclarePairedDelimiter{\ang}{\langle}{\rangle}

\DeclarePairedDelimiterX{\comm}[2]{\lbrack}{\rbrack}{#1 , #2}  
\DeclarePairedDelimiterX{\acomm}[2]{\lbrace}{\rbrace}{#1 , #2} 
\DeclarePairedDelimiterX{\inner}[2]{\langle}{\rangle}{#1 , #2} 
\DeclarePairedDelimiterX{\super}[2]{\lparen}{\rparen}{#1 \delimsize\vert \mathopen{} #2} 

\newcommand{\ra}{\rightarrow}
\newcommand{\lra}{\longrightarrow}
\newcommand{\ira}{\hookrightarrow}    
\newcommand{\sra}{\twoheadrightarrow} 

\newcommand{\dses}[5]{0 \lra #1 \overset{#2}{\lra} #3 \overset{#4}{\lra} #5 \lra 0} 

\newcommand{\blank}{{-}}

\newcommand{\fld}[1]{\mathbb{#1}}    
\newcommand{\alg}[1]{\mathfrak{#1}}  
\newcommand{\grp}[1]{\mathsf{#1}}    
\newcommand{\aalg}[1]{\mathsf{#1}}   
\newcommand{\Mod}[1]{\mathcal{#1}}   
\newcommand{\VOA}[1]{\mathsf{#1}}    
\newcommand{\categ}[1]{\mathscr{#1}}  
\newcommand{\funct}[1]{\mathscr{#1}} 

\newcommand{\ZZ}{\fld{Z}}
\newcommand{\NN}{\ZZ_{\ge 0}} 

\newcommand{\RR}{\fld{R}}
\newcommand{\CC}{\fld{C}}

\newcommand{\SLG}[2]{\grp{#1}_{#2}}            

\newcommand{\ab}{\alg{b}}
\newcommand{\ag}{\alg{g}}
\newcommand{\ah}{\alg{h}}

\newcommand{\al}{\alg{l}}
\newcommand{\ap}{\alg{p}}
\newcommand{\as}{\alg{s}}
\newcommand{\au}{\alg{u}}
\newcommand{\az}{\alg{z}}
\newcommand{\ahs}{\ah_{\as}}

\newcommand{\finite}[1]{#1}
\newcommand{\affine}[1]{\widehat{#1}}

\newcommand{\kmg}{\affine{\ag}}

\newcommand{\SLA}[2]{\finite{\alg{#1}}_{#2}}             
\newcommand{\SLSA}[3]{\finite{\alg{#1}} \super{#2}{#3}}  

\newcommand{\cox}{\mathsf{h}}                                        
\newcommand{\dcox}{\coroot{\cox}}                                    

\newcommand{\rootsetsymb}{\Delta}
\newcommand{\srootsetsymb}{\Pi}
\newcommand{\dynkinsymb}{\Gamma}

\newcommand{\cent}[2]{\envalg{#2}^{#1}}                              
\newcommand{\cas}{\Omega}                                            
\newcommand{\rlat}{\grp{Q}}                                          
\newcommand{\wlat}{\grp{P}}                                          
\newcommand{\roots}[1]{\rootsetsymb_{#1}}                            
\newcommand{\sroots}[1]{\srootsetsymb_{#1}}                          
\newcommand{\sroot}[1]{\alpha_{#1}}                                  
\newcommand{\hroot}{\theta}                                          
\newcommand{\wvec}{\rho}                                             
\newcommand{\fwt}[1]{\omega_{#1}}                                    
\newcommand{\coroot}[1]{#1^{\vee}}                                   
\newcommand{\scroot}[1]{\coroot{\sroot{#1}}}                         
\newcommand{\dynkin}[1]{\dynkinsymb_{#1}}                            

\newcommand{\bilin}[2]{\brac{#1 , #2}}                               
\newcommand{\wgrp}{\grp{W}}                                          
\newcommand{\swgrp}[1]{\wgrp_{#1}}                                   
\newcommand{\wref}[1]{w_{#1}}                                        
\newcommand{\swref}[1]{s_{#1}}                                       

\newcommand{\supp}[1]{\Supp(#1)}                                     
\newcommand{\suppess}[1]{\EssSupp(#1)}                               

\newcommand{\bounded}{\mathcal{B}}                                   

\newcommand{\injmonoid}[1]{\rootsetsymb^{\textup{inj.}}(#1)}         
\newcommand{\umonoid}{\rootsetsymb_{\au}^{\ge}}                      
\newcommand{\nilroots}[1]{\rootsetsymb^{\textup{nil.}}_{#1}}         
\newcommand{\nilrlat}[1]{\rlat^{\textup{nil.}}_{#1}}                 

\newcommand{\type}[1]{\textup{#1}}                                   

\newcommand{\aaa}{\aalg{A}}
\newcommand{\aab}{\aalg{B}}
\newcommand{\aai}{\aalg{I}}
\newcommand{\aaj}{\aalg{J}}
\newcommand{\aau}{\aalg{U}}
\newcommand{\aaz}{\aalg{Z}}

\newcommand{\envalg}[1]{\aau\brac{#1}}                          

\newcommand{\uaffvoa}[2]{\VOA{V}_{#1} \brac{#2}}                
\newcommand{\affvoa}[2]{\VOA{L}_{#1} \brac{#2}}                 
\newcommand{\vkg}{\uaffvoa{\kk}{\ag}}                           
\newcommand{\lkg}{\affvoa{\kk}{\ag}}                            
\newcommand{\slvoa}[2]{\affvoa{#1}{\SLA{sl}{#2}}}               
\newcommand{\sovoa}[2]{\affvoa{#1}{\SLA{so}{#2}}}               
\newcommand{\spvoa}[2]{\affvoa{#1}{\SLA{sp}{#2}}}               

\newcommand{\zhu}[1]{\operatorname{\mathsf{Zhu}}\sqbrac{#1}}             
\newcommand{\ind}[1]{\operatorname{\mathsf{Ind}}\sqbrac{#1}}             

\newcommand{\mdc}{\Mod{C}}
\newcommand{\mdd}{\Mod{D}}
\newcommand{\mdh}{\Mod{H}}
\newcommand{\mdl}{\Mod{L}}
\newcommand{\mdm}{\Mod{M}}
\newcommand{\mdn}{\Mod{N}}
\newcommand{\mdp}{\Mod{P}}

\newcommand{\pind}[1]{\funct{I}_{#1}} 
\newcommand{\uinv}[1]{\funct{R}_{#1}} 

\newcommand{\cft}{conformal field theory}
\newcommand{\cfts}{conformal field theories}
\newcommand{\uea}{universal enveloping algebra}

\newcommand{\lw}{lowest-weight}
\newcommand{\lwv}{\lw{} vector}

\newcommand{\lwms}{\lw{} modules}
\newcommand{\hw}{highest-weight}
\newcommand{\hwv}{\hw{} vector}
\newcommand{\hwvs}{\hw{} vectors}
\newcommand{\hwm}{\hw{} module}
\newcommand{\hwms}{\hw{} modules}
\newcommand{\sv}{singular vector}

\newcommand{\voa}{vertex operator algebra}
\newcommand{\voas}{vertex operator algebras}

\newcommand{\svoa}{vertex operator superalgebra}
\newcommand{\svoas}{vertex operator superalgebras}

\newcommand{\lhs}{left-hand side}

\newcommand{\rhs}{right-hand side}

\newcommand{\fdim}{finite-dimensional}
\newcommand{\infdim}{infinite-dimensional}

\newcommand{\wzw}{Wess--Zumino--Witten}

\newcommand{\pbw}{Poincar\'{e}--Birkhoff--Witt}
\newcommand{\bgg}{Bern\v{s}te\u{\i}n--Gel'fand--Gel'fand}

\theoremstyle{plain}
\newtheorem{mtheorem}{Main Theorem}
\newtheorem{theorem}{Theorem}[section]
\newtheorem{corollary}[theorem]{Corollary}
\newtheorem{lemma}[theorem]{Lemma}
\newtheorem{proposition}[theorem]{Proposition}

\newtheorem{definition}[theorem]{Definition}

\newtheorem*{algorithm}{Algorithm}

\makeatletter
\renewcommand\author@andify{%
  \nxandlist {\unskip ,\penalty-1 \space\ignorespaces}%
    {\unskip {} \@@and~}%
    {\unskip \penalty-2 \space \@@and~}%
}
\makeatother

\begin{document}

\title[Relaxed highest-weight modules II: classifications for affine vertex algebras]{Relaxed highest-weight modules II: \\ classifications for affine vertex algebras}

\author[K~Kawasetsu]{Kazuya Kawasetsu}
\address[Kazuya Kawasetsu]{
Priority Organization for Innovation and Excellence\\
Kumamoto University\\
Kumamoto, Japan, 860-8555.
}
\email{kawasetsu@kumamoto-u.ac.jp}

\author[D~Ridout]{David Ridout}
\address[David Ridout]{
School of Mathematics and Statistics \\
University of Melbourne \\
Parkville, Australia, 3010.
}
\email{david.ridout@unimelb.edu.au}

\begin{abstract}
	This is the second of a series of articles devoted to the study of relaxed \hwms{} over affine vertex algebras and W-algebras.  The first \cite{KawRel18} studied the simple ``rank-$1$'' affine vertex superalgebras $\slvoa{\kk}{2}$ and $\affvoa{\kk}{\SLSA{osp}{1}{2}}$, with the main results including the first complete proofs of certain conjectured character formulae (as well as some entirely new ones).  Here, we turn to the question of classifying relaxed \hwms{} for simple affine vertex algebras of arbitrary rank.  The key point is that this can be reduced to the classification of \hwms{} by generalising Olivier Mathieu's coherent families \cite{MatCla00}.  We formulate this algorithmically and illustrate its practical implementation with several detailed examples.  We also show how to use coherent family technology to establish the non-semisimplicity of category $\categ{O}$ in one of these examples.
\end{abstract}

\maketitle

\onehalfspacing

\section{Introduction} \label{sec:intro}

\subsection{Aims}
The representation theory of the \svoa{} underlying a given \cft{} is traditionally assumed to have a \hw{} flavour, especially when the theory in question is rational.  However, there is a generalisation that is playing an increasingly important role in studying non-rational examples, namely the \emph{relaxed} \hwms{}.  These were originally named in \cite{FeiEqu98} where such modules over the simple (admissible-level) affine \voa{} $\slvoa{\kk}{2}$ were used to study the well-known Kazama-Suzuki correspondence \cite{KazCha89} with the $N=2$ superconformal \svoas{}.

The idea behind the appellation ``relaxed'' comes from relaxing the definition of a \hwv{} so that it no longer needs to be annihilated by the positive root vectors of the horizontal subalgebra.  A relaxed \hwm{} is then just a module that is generated by a relaxed \hwv{}.  This idea can be applied to quite general classes of \svoas{} \cite{RidRel15} and so relaxed \hwms{} are potentially important ingredients of a wide variety of \cfts{}.

Interestingly, the simple relaxed \hw{} $\slvoa{\kk}{2}$-modules were actually classified in \cite{AdaVer95}, several years before their naming in \cite{FeiEqu98}.  They have since been proposed as the main building blocks of the $\SLG{SL}{2}(\RR)$ \wzw{} models \cite{MalStr01}, found to arise naturally in the fusion rules of $\slvoa{-4/3}{2}$ and $\slvoa{-1/2}{2}$ \cite{GabFus01,RidFus10}, and used to analyse the representation theory of the admissible-level $\SLA{sl}{2}$-parafermion theories \cite{AdaCon05,RidSL210,CreCos13,AugMod17}.  Moreover, relaxed \hwms{} have recently been shown to play a central role in \cfts{} based on the \svoas{} $\slvoa{\kk}{3}$ \cite{AdaRea16,AraWei16,KawAdm19}, $\affvoa{\kk}{\SLSA{osp}{1}{2}}$ \cite{RidAdm17,AdaRea17,CreCos18} and $\affvoa{\kk}{\SLSA{sl}{2}{1}}$ \cite{CreAdm19}.

One of the many reasons to study relaxed \hwms{} is the belief that such modules are necessary to construct consistent affine \cfts{} at non-rational levels.  Indeed, it has been observed in several examples \cite{CreMod12,CreMod13,RidBos14,RidAdm17,CreCos18,KawAdm19} that the characters of the representations of a \svoa{} need not carry a representation of the modular group unless one includes relaxed modules (and their twists by spectral flow automorphisms \cite{SchCom87,RidSL208}).  Further, this inclusion even allows one, in these cases, to compute the Grothendieck fusion coefficients using a (conjectural) Verlinde formula \cite{CreLog13,RidVer14}.

From the point of view of this article, however, the most compelling reason to study relaxed \hwms{} is the fact that they form the largest class of weight modules to which Zhu's powerful classification methods \cite{ZhuMod96} may be applied.  More precisely, the simple relaxed \hwms{} are the simple objects of a relaxed category $\categ{R}$, see \cite{RidRel15,KawRel18} for the definition, that naturally generalises the well-known \bgg{} category $\categ{O}$.  The point is that this is the largest category of weight modules on which Zhu's functor $\zhu{\blank}$ (introduced below) has zero kernel.

Our aim here is to provide the means to classify the simple relaxed \hwms{}, with \fdim{} weight spaces, over an \emph{arbitrary} affine \voa{}.  The restriction to \fdim{} weight spaces is motivated physically by the need to have well-defined characters, in particular so that the modular invariance of the partition function of the \cft{} can be verified.  Actually, the method works for critical levels as well where one also expects relaxed modules, see \cite{AdaLie07} for the $\SLA{sl}{2}$ case.  To the best of our knowledge, relaxed classifications are currently only known for $\slvoa{\kk}{2}$ \cite{AdaVer95,RidRel15}, $\affvoa{\kk}{\SLSA{osp}{1}{2}}$ \cite{RidAdm17,WooAdm18,CreCos18} and $\slvoa{\kk}{3}$ \cite{AraWei16}.  Our results make it easy to extend these classifications to higher-rank $\ag$, at least when the level $\kk$ is admissible \cite{KacMod88}, thanks to the celebrated \hw{} classification of Arakawa \cite{AraRat16}.

\subsection{Zhu technology} \label{sec:zhu}

Throughout this work, the underlying field is always implicitly assumed to be the complex numbers $\CC$.  Let $\ag$ be a \fdim{} simple Lie algebra.  Recall that the Zhu algebra \cite{ZhuMod96} of a level-$\kk$ affine vertex algebra $\vkg$ is isomorphic \cite{FreVer92} to
\begin{equation}
	\aaz_{\kk} = \frac{\envalg{\ag}}{\aai_{\kk}},
\end{equation}
where $\aai_{\kk}$ is some two-sided ideal of $\envalg{\ag}$.  If $\vkg$ is universal, then $\aai_{\kk} = 0$.  If $\vkg$ is not universal, then $\aai_{\kk}$ is non-zero if and only if $\kk$ is critical, meaning that $\kk=-\dcox$, or $\kk$ satisfies \cite{GorSim07}
\begin{equation} \label{eq:simplevacuum}
	\ell (\kk + \dcox) = \frac{u}{v}, \qquad \text{for some}\ u \in \ZZ_{\ge 2}\ \text{and}\ v \in \ZZ_{\ge 1}\ \text{with}\ \gcd \set{u,v} = 1.
\end{equation}
Here, $\ell$ is the lacing number of $\ag$:  $\ell = 1$ for types \type{A}, \type{D} and \type{E}; $\ell = 2$ for types \type{B}, \type{C} and \type{F}; $\ell = 3$ for type \type{G}.

The representation theories of a vertex superalgebra and its Zhu algebra are related \cite{ZhuMod96} by a functor $\zhu{\blank}$.  For an affine vertex algebra $\vkg$, this functor maps the category of relaxed \hw{} $\vkg$-modules to the category of weight $\aaz_{\kk}$-modules.  If we recall \cite{FreVer92} that any $\vkg$-module is naturally a module over the untwisted affine Kac-Moody algebra $\kmg = \ag[t,t^{-1}] \oplus \CC K$ (on which the central element $K$ acts as multiplication by $\kk$), then this functor has the form $\zhu{\mdm} = \mdm^{t \ag[t]}$ (the elements of $\mdm$ that are annihilated by $t \ag[t]$).

There is likewise a functor $\ind{\blank}$ from the category of weight $\aaz_{\kk}$-modules to the category of relaxed \hw{} $\vkg$-modules, obtained by ``inducing'' and then quotienting by the maximal submodule whose intersection with the original module is zero.  We refer to \cite[Sec.~2.2]{ZhuMod96} and \cite[Sec.~3.2]{LiRep94} for a precise definition of what ``inducing'' means in this context.  Using these two functors, Zhu proved the following celebrated result (actually in much greater generality).
\begin{theorem}[Zhu \protect{\cite[Thms.~2.2.1 and 2.2.2]{ZhuMod96}}] \label{thm:zhu}
	\leavevmode
	\begin{enumerate}
		\item A relaxed \hw{} $\vkg$-module $\mdl$ is simple if and only if $\zhu{\mdl}$ is a simple weight $\aaz_{\kk}$-module. \label{it:zhusimp}
		\item More generally, any $\aaz_{\kk}$-module $\mdm$ yields an $\NN$-graded $\vkg$-module $\ind{\mdm}$ such that $\zhu{\ind{\mdm}} \cong \mdm$ and $0$ is the only submodule of $\ind{\mdm}$ whose intersection with the ``top space'' $\mdm$ is zero. \label{it:zhuindec}
	\end{enumerate}
\end{theorem}
\noindent To classify the simple relaxed \hwms{} of the affine vertex algebra $\vkg$, it therefore suffices to classify the simple weight modules of $\ag$ that are annihilated by the Zhu ideal $\aai_{\kk}$.

If $\aai_{\kk} = 0$, which occurs when $\vkg$ is universal, our task is then to classify all the weight modules of $\ag$.  This is quite ambitious and has in fact only been completed for $\ag = \SLA{sl}{2}$ (see \cite{MazLec10} for a textbook treatment).  However, as noted above, we actually want to restrict to weight modules with \fdim{} weight spaces.  Then, we are in better shape because this class of $\ag$-modules was classified, for all finite-dimensional simple Lie algebras $\ag$, by Mathieu \cite{MatCla00} (building on work of Fernando \cite{FerLie90}).  For this purpose, Mathieu introduced highly reducible $\ag$-modules called coherent families whose properties reduced the classification problem to the classification of \hw{} $\ag$-modules satisfying certain easily analysed conditions.

In this paper, we are interested in the case in which $\aai_{\kk} \neq 0$.  We will therefore extend Mathieu's result to a classification of all simple weight $\aaz$-modules with \fdim{} weight spaces, where $\aaz$ is the quotient of $\envalg{\ag}$ by an arbitrary two-sided ideal $\aai$.  More precisely, we use Mathieu's theory of coherent families to reduce this classification to a classification problem involving \hw{} $\aaz$-modules.  In particular, if the classification of simple \hw{} $\aaz$-modules is already known, then our results allow one to algorithmically classify all the simple weight $\aaz$-modules with \fdim{} weight spaces (see \cref{sec:algorithm} and the examples detailed in \cref{sec:ex}).  Specialising $\aaz$ to the Zhu algebra $\aaz_{\kk}$ of $\vkg$ and applying \cref{thm:zhu}\ref{it:zhusimp}, we then recover the relaxed \hw{} classification that we are interested in here.

\subsection{Results} \label{sec:results}

In this \lcnamecref{sec:results}, we present our results in the context of classifying certain types of relaxed \hw{} $\vkg$-modules with finite-dimensional weight spaces.  As mentioned above, these results actually hold for ideals more general than the Zhu ideals $\aai_{\kk}$ and are stated as such in the rest of the paper.

Before stating our main theorems, we shall need to introduce some definitions.  First, we generalise Mathieu's notion \cite{MatCla00} of a coherent family of $\ag$-modules to families of $\al$-modules, where $\al$ is an arbitrary \fdim{} reductive Lie algebra.  Fixing a Cartan subalgebra $\ah$ of $\al$, we let $\as = \comm{\al}{\al}$ and $\ah_{\as} = \ah \cap \as$.  Then, a \emph{coherent family} of $\al$-modules is a weight $\al$-module satisfying the following three properties: its (weight) support is a single coset $\zeta + \ah_{\as}^*$ (for some $\zeta \in \ah^*$); its non-zero weight spaces all have the same dimension; any element of the centraliser of $\ah$ in $\envalg{\al}$ defines a polynomial function on the support given by the trace of the element's action on each weight space.  We refer to \cref{def:coh} below for further discussion.

The reason why we need this minor generalisation of coherent families is that we require a further generalisation that also accounts for Fernando's work \cite{FerLie90}.  For this, we consider parabolic subalgebras $\ap \subseteq \ag$ and take $\al$ to be the corresponding Levi factor.  Parabolic induction then defines a functor that maps a weight $\al$-module to a weight $\ag$-module, canonically embedding the former in the latter.  We define the \emph{almost-simple quotient} of a parabolically induced module to be the quotient by the sum of all the submodules that have zero intersection with the image of this embedding.

With this, we can finally define the promised generalisation of coherent families: a \emph{parabolic family} of $\ag$-modules is the almost-simple quotient of the parabolic induction of some coherent family of $\al$-modules, see \cref{def:par}.  One useful property of a coherent family $\mdc$ of $\al$-modules is that it always contains an \emph{infinite-dimensional \hw} submodule $\mdh$ \cite{MatCla00}.  The almost simple quotient of the parabolic induction of $\mdh$ is then an infinite-dimensional \hw{} $\ag$-submodule of the parabolic family induced from $\mdc$.  We shall refer to the \hw{} $\ag$-modules obtained in this fashion as being \emph{$\al$-bounded}, referring to \cref{def:lbound} below for further details.  Note that not every infinite-dimensional \hw{} submodule of a parabolic family is automatically $\al$-bounded.

We can now present our first main theorem.  Recall that $\ag$ denotes a \fdim{} simple Lie algebra, $\kmg$ its untwisted affinisation and $\vkg$ one of the corresponding affine vertex algebras of level $\kk \in \CC$.
\begin{mtheorem} \label{thm:zhuclass}
	Suppose that $\mdl$ is a simple level-$\kk$ relaxed \hw{} $\kmg$-module, with \fdim{} weight spaces, that is not \hw{} with respect to any Borel subalgebra.  Then, $\mdl$ is a $\vkg$-module if and only if $\zhu{\mdl}$ is a submodule of an irreducible semisimple parabolic family $\mdp$ of $\ag$-modules that has a simple \emph{$\al$-bounded \hw} submodule $\mdh$ whose Zhu-induction $\ind{\mdh}$ is a $\vkg$-module.  Here, $\al$ denotes the Levi factor of the parabolic subalgebra associated with $\mdp$.
\end{mtheorem}
\noindent The notion of irreducibility and semisimplicity for parabolic families is defined in \cref{sec:para}, see \cref{eq:decompP}.

This result follows immediately by combining \cref{thm:zhu}\ref{it:zhusimp} with \cref{thm:classification} below.  What it means is that if one is able to classify the simple \hw{} $\vkg$-modules and understand the \hw{} submodules of every parabolic family of $\kmg$-modules, then one can deduce the classification of the simple \emph{relaxed} \hw{} $\vkg$-modules.  We shall see how this works with a series of examples in \cref{sec:ex}.

Our second main theorem extends the first to cover certain types of non-simple, but indecomposable, relaxed \hw{} $\vkg$-modules.  Given a root $\alpha$ of $\ag$, we say that a $\ag$-module $\mdm$ is \emph{$\alpha$-bijective} if the corresponding root vector acts bijectively.
\begin{mtheorem} \label{thm:zhuind}
	Let $\ap \subseteq \ag$ be a parabolic subalgebra of $\ag$ with a non-abelian Levi factor $\al$ and let $\mdc$ be an irreducible $\alpha$-bijective coherent family of $\al$-modules, for some root $\alpha$ of $\al$.  Let $\mdp$ denote the parabolic family of $\ag$-modules induced from $\mdc$ and let $\mdh$ be a simple \emph{$\al$-bounded \hw} submodule of $\mdp$.  Then, if $\ind{\mdh}$ is a $\vkg$-module, then so is every subquotient of $\ind{\mdp}$.
\end{mtheorem}
\noindent This result follows from \cref{thm:zhu}\ref{it:zhuindec} and \cref{thm:parabolicindecomposable}. The condition of $\alpha$-bijectivity ensures that $\mdp$ is not semisimple, hence that it has reducible but indecomposable subquotients from which we obtain reducible but indecomposable $\vkg$-modules by Zhu-induction.  Of course, identifying these indecomposable $\vkg$-modules may be quite difficult in practice.  In \cref{sec:catO}, we consider an illustrative application that features a non-semisimple parabolic family of $\SLA{so}{8}$-modules.

We mention that the motivation for wanting to construct such non-simple indecomposable relaxed \hw{} $\vkg$-modules stems from the observation \cite{SalGL106,RidFus10} that such modules seem to be building blocks for constructing projective covers (in a category that naturally extends the relaxed category $\categ{R}$ by spectral flow).  These projective covers are, in turn, believed to be the natural building blocks of the state space of the \cft{} \cite{SalGL106,CreLog13}.  Unfortunately, these covers are currently not even known to exist for any non-rational affine theory, though conjectural structures for $\slvoa{\kk}{2}$ and $\affvoa{\kk}{\SLSA{osp}{1}{2}}$ may be found in \cite{CreCos18,CreUni19,LiuCos19}.

\subsection{Outline}

We start by recalling Mathieu's definition of a coherent family of $\ag$-modules in \cref{sec:coh} and by immediately generalising it to coherent families of modules over a reductive Lie algebra $\al$.  This \lcnamecref{sec:coh} also introduces some convenient definitions and summarises some of the important results of Fernando and Mathieu that are needed in what follows.  \cref{sec:para} then introduces a new notion, which we call a parabolic family of $\ag$-modules, and formalises the relationship between parabolic and coherent families in terms of restriction- and induction-type functors.

The classification work begins in \cref{sec:class}.  For a quotient $\aaz$ of $\envalg{\ag}$ by an arbitrary ideal, we identify the simple weight $\aaz$-modules, with \fdim{} weight spaces, as simple submodules of certain semisimple parabolic families of $\ag$-modules (\cref{thm:classification}).  This proves \cref{thm:zhuclass}.  The extension to $\alpha$-bijective indecomposable modules, needed for \cref{thm:zhuind}, is then proven in \cref{sec:indec} (\cref{thm:parabolicindecomposable}), now using non-semisimple parabolic families.

Having proven these classification theorems, we next turn to the question of how to efficiently analyse the combinatorics of parabolic families so as to be able to exploit existing \hw{} classification results.  For this, we first summarise Mathieu's explicit classification of coherent families in \cref{sec:ac}.  Interestingly, it turns out that coherent families are usually, but not always, completely distinguished by their central characters.  \cref{sec:refine} then describes when two \hwms{} appear as submodules of the same coherent/parabolic family and discusses how the Weyl group acts on parabolic families.

This material is combined with \cref{thm:classification} in \cref{sec:algorithm} and the result is summarised in terms of an algorithm for classifying simple weight $\aaz$-modules with \fdim{} weight spaces.  In \cref{sec:ex}, we use this algorithm to classify the simple relaxed \hwms{} of some interesting examples, taking $\aaz$ to be the Zhu algebra $\aaz_{\kk}$ of a simple affine \voa{} $\lkg$.  Specifically, we address the admissible-level cases $\slvoa{-3/2}{3}$, $\spvoa{-1/2}{4}$ and $\affvoa{-5/3}{\ag_2}$ as well as the non-admissible-level case $\sovoa{-2}{8}$.  We hope that these illustrations will provide the reader with a taste of the utility of our results.

Finally, we give an application of the utility of \cref{thm:parabolicindecomposable} in \cref{sec:catO}.  Specifically, we use it to show that the simple affine \voa{} $\sovoa{-2}{8}$ not only admits non-semisimple relaxed \hwms, but it in fact also admits non-semisimple \hwms.  We believe that this is the first demonstration of non-semisimplicity in category $\categ{O}$ for a \emph{quasilisse} \cite{AraQua16} affine \voa.

In the future, we intend to explore more families of higher-rank classifications in order to better understand the general features of relaxed \hwms{}.  We also intend to generalise the methodology developed here to affine vertex superalgebras and the associated W-algebras and superalgebras.  Note that there are many interesting cases \cite{AdaCla03,AdaLog07,AdaN=109,CreRel11,CreWAl11,BabTak12,CreFal13,RidMod13,MorKac15,CanFusI15,CanFusII15,SatEqu16,AugMod17,CreUni19} in which a vertex algebra possesses continuously parametrised ``coherent'' families consisting of \emph{highest-weight} modules.  We also hope to generalise our treatment of weight modules so as to study these cases.  The next instalment of this series \cite{KawRel20} will address the important problem of computing the character of more general relaxed \hwms{}, thus generalising the rank-$1$ results of \cite{KawRel18}.

\section*{Acknowledgements}

We thank Dra\v{z}en Adamovi\'{c}, Tomoyuki Arakawa, Thomas Creutzig, Terry Gannon, Kenji Iohara, Masoud Kamgarpour, Olivier Mathieu, Walter Mazorchuk, Jethro van Ekeren and Simon Wood for useful discussions relating to the material presented here and for their encouragement.
We likewise thank the reviewer for their careful and very useful report.
KK's research is partially supported by the Australian Research Council Discovery Project DP160101520, JSPS KAKENHI Grant Number 19J01093 and  19KK0065 and MEXT Japan ``Leading Initiative for Excellent Young Researchers (LEADER)''.
DR's research is supported by the Australian Research Council Discovery Project DP160101520 and the Australian Research Council Centre of Excellence for Mathematical and Statistical Frontiers CE140100049.

\section{Coherent families} \label{sec:coh}

In \cite{MatCla00}, Olivier Mathieu introduced the notion of a coherent family as a fundamental tool for completing the classification of simple weight modules with \fdim{} weight spaces over a \fdim{} simple Lie algebra $\ag$.  Fix a Cartan subalgebra $\ah \subseteq \ag$.  We let $\supp{\mdm} \subseteq \ah^*$ denote the \emph{support} (the set of weights) of a module $\mdm$ and write $\mdm(\mu)$ for the weight space of $\mdm$ corresponding to the weight $\mu \in \ah^*$.  Let $\cent{\ah}{\ag}$ denote the centraliser of $\ah$ in the \uea{} $\envalg{\ag}$.  Mathieu's definition is then as follows.
\begin{definition}
	Let $\ag$ be a \fdim{} simple Lie algebra.  A \emph{coherent family} of $\ag$-modules is a weight $\ag$-module $\mdc$ for which:
	\begin{itemize}
		\item There exists $d \in \ZZ_{>0}$, called the \emph{degree} of $\mdc$, such that $\dim \mdc(\mu) = d$ for all $\mu \in \ah^*$.
		\item Given any $U \in \cent{\ah}{\ag}$, the function taking $\mu \in \ah^*$ to $\tr_{\mdc(\mu)} U$ is polynomial in $\mu$.
	\end{itemize}
\end{definition}
\noindent In particular, the support of a coherent family is all of $\ah^*$.

We shall need analogues of these families for certain \fdim{} \emph{reductive} Lie algebras $\al$, each also coming with a fixed Cartan subalgebra $\ah \subseteq \al$.  We let $\as = \comm{\al}{\al}$ denote the derived subalgebra of $\al$ and choose a Cartan subalgebra of $\as$ to be $\ahs = \ah \cap \as$.  We then have $\al = \as \oplus \az$ and $\ah = \ahs \oplus \az$, where $\az$ is the centre of $\al$.
\begin{definition} \label{def:coh}
	Let $\al$ be a \fdim{} reductive Lie algebra.  A \emph{coherent family} of $\al$-modules is a weight $\al$-module $\mdc$ for which:
	\begin{itemize}
		\item $\supp{\mdc} = \zeta + \ahs^*$, for some $\zeta \in \ah^*$.
		\item There exists $d \in \ZZ_{>0}$, called the \emph{degree} of $\mdc$, such that $\dim \mdc(\mu) = d$ for all $\mu \in \supp{\mdc}$.
		\item Given any $U \in \cent{\ah}{\al}$, the function taking $\mu \in \supp{\mdc}$ to $\tr_{\mdc(\mu)} U$ is polynomial in $\mu$.
	\end{itemize}
\end{definition}
\noindent This reduces to Mathieu's definition when $\al$ is simple.

This reduction of the support from $\ah^*$ to $\zeta + \ahs^*$ is motivated by the idea that a given polynomial action on a suitable \emph{\infdim{}} submodule automatically determines the action on the entire coherent family.  As we shall see, the simple ideals of $\al$ may have \infdim{} submodules that can be used for such purposes, while the abelian ideal $\az$ of course does not.

A coherent family $\mdc$ of $\al$-modules is therefore highly reducible in general, decomposing as
\begin{equation} \label{eq:decompC}
	\mdc \cong \bigoplus_{\lambda \in \supp{\mdc} / \rlat_{\al}} \mdc_{\lambda},
\end{equation}
where $\rlat_{\al}$ denotes the root lattice of $\al$ (which coincides with that of $\as$).  If at least one of the $\mdc_{\lambda}$ is simple, then the coherent family $\mdc$ is said to be \emph{irreducible}.  Likewise, $\mdc$ is called a \emph{semisimple} coherent family if all of the $\mdc_{\lambda}$ are semisimple $\al$-modules.

We note the special case in which $\al$ is abelian, hence $\az = \ah = \al$, $\as = \ahs = 0$ and $\cent{\ah}{\al} = \envalg{\az}$.  Then, the support of a coherent family $\mdc$ of $\al$-modules is a singleton $\set{\zeta}$.  It follows that $\mdc \cong \mdc_{\zeta}$ is a (possibly non-semisimple) extension of $d$ copies of the simple $\az$-module of weight $\zeta$.  It is clear that the trace of the action of each $U \in \envalg{\az}$ amounts to multiplication by $d \zeta(U)$, where $d$ is the degree of $\mdc$.  This is clearly polynomial in $\zeta$.  If $\mdc$ is irreducible, then it is automatically semisimple with degree $d=1$.  Indeed, in this case, $\mdc$ is actually simple as a $\az$-module.

A somewhat less trivial example is $\al = \SLA{sl}{2}$ for which $\cent{\ah}{\al}$ is the polynomial ring generated by $\ah$ and the centre $\cent{\al}{\al}$, the latter being polynomials in the quadratic Casimir $\cas$.  The classification of simple weight modules (with \fdim{} weight spaces) is therefore elementary, see \cite[Thm.~3.32]{MazLec10} for example.  Indeed, a simple weight module is either \hw{}, \lw{}, or dense, where we recall that a weight $\al$-module $\mdn$ is said to be \emph{dense} if $\supp{\mdn} = \lambda + \rlat_{\al}$, for some $\lambda \in \ah^*$.  The summands $\mdc_{\lambda}$ of an irreducible semisimple coherent family $\mdc$ over $\SLA{sl}{2}$ are thus either direct sums of simple highest- and \lwms{} or are simple and dense.  Moreover, the latter case is generic, occurring whenever there are no $\mu \in \lambda + \rlat_{\al}$ satisfying the \hw{} condition relating $\mu$ to the eigenvalue of $\cas$.  Note that $\cas$ acts as a constant on each simple summand of $\mdc$, by Schur's lemma, hence it must act as a constant on all of $\mdc$ in order to act polynomially.

We consider one last example: $\al = \SLA{gl}{2}$, for which we have $\az \cong \SLA{gl}{1}$ and $\as \cong \SLA{sl}{2}$.  A simple weight $\al$-module is therefore a \hw{}, \lw{} or dense $\as$-module tensored by a one-dimensional $\az$-module.  Our definition for an irreducible degree-$d$ coherent family $\mdc$ of $\SLA{gl}{2}$-modules is now seen to reduce to the tensor product of an irreducible degree-$d$ coherent family of $\SLA{sl}{2}$-modules with a fixed simple $\SLA{gl}{1}$-module.  Indeed, if $\supp{\mdc} = \zeta + \ahs^*$, then one may choose $\zeta \in \az^* \subset \ah^*$ to be the unique weight of the fixed $\SLA{gl}{1}$-module.

The picture for irreducible $\SLA{sl}{2}$ (and $\SLA{gl}{2}$) coherent families $\mdc$ is then that they decompose into direct summands $\mdc_{\lambda}$ that are simple and dense for all but a finite number of $\lambda \in \supp{\mdc} / \rlat_{\al}$.  The non-simple summands have highest- and \lw{} composition factors that share their central character ($\cas$-eigenvalue) with the simple summands.  Unfortunately, this picture only generalises partially to higher ranks.  We prepare some convenient terminology.
\begin{definition}
	A \fdim{} reductive Lie algebra is said to be \emph{of \type{AC}-type} if its simple ideals are all of types \type{A} and \type{C}.
\end{definition}
\noindent We recall that the \emph{type} of a \fdim{} simple Lie algebra refers to the name given to its Dynkin diagram.  Thus, $\SLA{sl}{n}$ is of type \type{A} while $\SLA{sp}{2n}$ is of type \type{C}, for all $n \in \ZZ_{\ge2}$.  For our purposes, it is convenient to regard $\SLA{sl}{2} \cong \SLA{sp}{2}$ as being of type \type{A} only (see \cref{sec:ac}).
\begin{proposition}[Fernando \protect{\cite[Thm.~5.2 and Rem.~5.4]{FerLie90}}] \label{prop:FernandoAC}
	A \fdim{} reductive Lie algebra $\al$ admits a simple dense module if and only if it is of \type{AC}-type.
\end{proposition}

Despite this, coherent families provide the means to construct and understand simple weight modules with \fdim{} weight spaces, as we shall discuss below (see \cref{thm:fernando}).  First, we collect some useful definitions.
\begin{definition}
	\leavevmode
	\begin{itemize}
		\item A \emph{bounded} $\al$-module is an \infdim{} weight module for which there is a (finite) upper bound on the multiplicities (the dimensions of the weight spaces).  The maximal multiplicity is called the \emph{degree} of the $\al$-module.
		\item The \emph{essential support} $\suppess{\mdn}$ of a bounded $\al$-module $\mdn$ is the set of weights whose multiplicities are maximal.
	\end{itemize}
\end{definition}
\noindent We remark that Mathieu calls a weight module with uniformly bounded multiplicities \emph{admissible}.  We prefer not to use this terminology as it clashes, in our intended application, with a similar widely used terminology for certain affine vertex algebras and their modules \cite{KacMod88}.

We note that the simple weight $\SLA{sl}{2}$-modules (with \fdim{} weight spaces) are all bounded.  However, this situation is not typical: for example, a Verma module of a \fdim{} simple Lie algebra $\ag$ is bounded if and only if $\ag = \SLA{sl}{2}$.  On the other hand, a simple dense $\al$-module is \emph{torsion-free} \cite{FerLie90}, meaning that the root vectors of $\al$ act injectively, hence its (non-zero) multiplicities are constant.  Simple dense modules for $\al \neq \ah$ are thus always bounded.

We conclude this \lcnamecref{sec:coh} by quoting some fundamental results for coherent families, proofs for all of which may be found in Mathieu's article \cite{MatCla00}.  In fact, we present adaptations of Mathieu's results which apply to coherent families for \fdim{} reductive Lie algebras.  These adaptations are quite straightforward and follow immediately from the standard decomposition of a reductive Lie algebra into its simple and abelian ideals, as noted in \cite[Sec.~1]{MatCla00}.  The case where the Lie algebra is abelian is excluded for simplicity.
\begin{proposition}[Mathieu \cite{MatCla00}] \label{thm:mathieu1}
	Let $\al$ be a \fdim{} non-abelian reductive Lie algebra.  Then:
	\begin{enumerate}
		\item \label{it:zariski} [Prop.~3.5ii] The essential spectrum of a simple bounded $\al$-module is Zariski-dense in $\zeta + \ahs^*$, for some $\zeta \in \ah^*$.
		\item \label{it:cohextension} [Prop.~4.8i] Every simple bounded $\al$-module embeds into a unique irreducible semisimple coherent family.
		\item \label{it:degrees} [Prop.~4.8ii] Every \infdim{} submodule of an irreducible coherent family of degree $d$ is bounded and its degree is also $d$.
		\item \label{it:existence} [Lem.~5.3ii] Coherent families exist if and only if $\al$ is of \type{AC}-type (compare \cref{prop:FernandoAC}).
		\item \label{it:hwsubmod} [Prop.~5.7] Given an irreducible semisimple coherent family, there is a choice of Borel subalgebra for $\al$ such that the family contains a simple bounded \hwm{}.
	\end{enumerate}
\end{proposition}

\section{Parabolic families} \label{sec:para}

Let $\ag$ be a \fdim{} simple Lie algebra and $\ap \subseteq \ag$ be a parabolic subalgebra.  We choose, once and for all, a Cartan subalgebra $\ah$ for $\ag$ and restrict the parabolics we consider to always contain $\ah$.  Let $\au$ denote the nilradical of $\ap$, $\al = \ap / \au$ its Levi factor and $\au^-$ the nilradical opposite to $\au$, so that $\ag = \au^- \oplus \al \oplus \au$ (as vector spaces).  We denote the derived subalgebra of $\al$ by $\as$ and let $\ahs = \ah \cap \as$.  Finally, let $\az$ be the centre of $\al$ so that $\al = \as \oplus \az$ and $\ah = \ahs \oplus \az$.

Given a choice of Borel, hence a set of simple/positive roots for $\ag$, the parabolics containing the Borel are in bijection with the subsets of the set of simple roots.  In particular, such a subset $S$ defines $\al$ and $\au$ as follows: $\al$ is spanned by $\ah$ and the root vectors whose roots are integer linear combinations of the elements of $S$, while $\au$ is spanned by all the remaining positive root vectors.  A useful consequence that we shall use several times is that the root lattice of $\al$ (and $\as$) has zero intersection with the monoid $\umonoid$ generated by the roots whose root vectors span $\au$.

If $\ap$ is a Borel subalgebra of $\ag$, then $\al = \az = \ah$, $\as = 0$ and $\au$ is the nilradical of the Borel.  This corresponds to taking $S = \varnothing$.  At the other extreme, taking $S$ to be the set of all simple roots corresponds to $\ap = \ag$, whence $\al = \as = \ag$ and $\au = 0$.  A useful motivating example is that of $\ag = \SLA{sl}{3}$ and $\abs{S} = 1$.  This leads to $6$-dimensional parabolics $\ap$ with $\al \cong \SLA{gl}{2}$, hence $\as \cong \SLA{sl}{2}$ and $\az \cong \SLA{gl}{1}$, while $\au$ is spanned by two commuting root vectors.

Given a parabolic $\ap \subseteq \ag$ as above, there are two important functors that relate weight modules over $\ag$ and $\al$.  First, there is the restriction $\uinv{\ap}$ that maps a weight $\ag$-module $\mdm$ to its subspace $\mdm^{\au}$ of vectors annihilated by $\au$.  As $\comm{\au}{\al} \subseteq \au$, $\uinv{\ap} \mdm = \mdm^{\au}$ is naturally an $\al$-module.  To introduce the second functor, recall that the \emph{parabolic induction} of a weight $\al$-module $\mdn$ is defined to be the $\ag$-module that results from letting $\au$ act as $0$ and then inducing as $\envalg{\ag} \otimes_{\envalg{p}} \mdn$.  If $\mdn$ is simple, then its parabolic induction will have a unique simple quotient.  In general, we define the \emph{almost-simple quotient} of the parabolic induction of a weight $\al$-module to be the $\ag$-module obtained by quotienting by the sum of all modules that have zero intersection with the subspace $\wun \otimes \mdn$.  We denote by $\pind{\ap}$ the functor on a weight $\al$-module that first parabolically induces to a $\ag$-module and then replaces the result by its almost-simple quotient.

\begin{proposition} \label{prop:functors}
	The functors $\pind{\ap}$ and $\uinv{\ap}$ satisfy the following properties:
	\begin{enumerate}
		\item \label{it:functlex} $\uinv{\ap}$ is inclusion-preserving and maps simple weight $\ag$-modules to simple weight $\al$-modules (or $0$).
		\item \label{it:functrex} $\pind{\ap}$ maps simple weight $\al$-modules to simple weight $\ag$-modules.
		\item \label{it:functinv} $\uinv{\ap} \pind{\ap} \mdn \cong \mdn$, for all weight $\al$-modules $\mdn$.
		\item \label{it:functsimp} If $\mdn$ is a simple weight $\al$-module that embeds in a weight $\al$-module $\mdn'$, then $\pind{\ap} \mdn$ embeds in $\pind{\ap} \mdn'$.
	\end{enumerate}
\end{proposition}
\begin{proof}
	We first prove part \ref{it:functlex}.  The fact that $\uinv{\ap}$ preserves inclusions is clear.  Suppose then that $\mdm$ is a simple weight $\ag$-module with $\uinv{\ap} \mdm \neq 0$ and that $v_1$ and $v_2$ are (non-zero) weight vectors in $\uinv{\ap} \mdm \subseteq \mdm$.  Since $v_1$ and $v_2$ are annihilated by $\au$, \pbw{} and simplicity imply that $\envalg{\au^- \oplus \al} \cdot v_i = \envalg{\ag} \cdot v_i = \mdm$, for $i=1,2$.  In particular, $v_1 = U_1 v_2$ and $v_2 = U_2 v_1$ for some $U_1, U_2 \in \envalg{\au^- \oplus \al}$.  But, $v_1 = U_1 U_2 v_1$ requires that $U_1, U_2 \in \envalg{\al}$ because the $\NN$-span of the roots of $\au^-$, which are all negative with respect to an appropriate Borel, have zero intersection with the root lattice of $\al$.  Thus, $v_1$ and $v_2$ lie in the same $\al$-submodule of $\uinv{\ap} \mdm$, proving that the latter is simple.

	For part \ref{it:functrex}, suppose that $\mdn$ is a simple weight $\al$-module.  Since any non-zero $w \in \mdn$ is cyclic, $\wun \otimes w$ (or rather its image in the almost-simple quotient) must generate $\pind{\ap} \mdn$.  However, the submodule generated by any non-zero $v \in \pind{\ap} \mdn$ must contain an element of the form $\wun \otimes w$, for some $w \in \mdn$, because otherwise its intersection with $\wun \otimes \mdn$ would be zero.  This submodule is thus $\pind{\ap} \mdn$, proving that the latter is simple.

	To prove part \ref{it:functinv}, first note that $\mdn \cong \wun \otimes \mdn \subseteq \uinv{\ap} \pind{\ap} \mdn$ because $\au \cdot (\wun \otimes \mdn) = 0$.  If this inclusion were strict, then \pbw{} would imply that there exists a non-zero $v \in \au^- \envalg{\au^-} \otimes \mdn$ with $\au \cdot v = 0$.  Since $\comm{\al}{\au^-} \subseteq \au^-$, we would have
	\begin{equation}
		\envalg{\ag} \cdot v = \envalg{\au^- \oplus \al} \cdot v \subseteq \au^- \envalg{\au^-} \otimes \mdn.
	\end{equation}
	But then, $v$ would generate a non-zero submodule of $\pind{\ap} \mdn$ whose intersection with $\wun \otimes \mdn$ is zero, a contradiction.  We therefore conclude that the inclusion is an equality.

	Finally, inducing from weight $\ap$-modules to $\ag$-modules is exact, by \pbw{}.  The sum of the submodules of $\envalg{\ag} \otimes_{\envalg{\ap}} \mdn$ whose intersection with $\wun \otimes \mdn$ is zero therefore embeds into the sum of the submodules of $\envalg{\ag} \otimes_{\envalg{\ap}} \mdn'$ whose intersection with $\wun \otimes \mdn'$ is zero, so it follows that we have a morphism from $\pind{\ap} \mdn$ to $\pind{\ap} \mdn'$.  This morphism is non-zero since it is non-zero on $\wun \otimes \mdn$, hence it is injective by the simplicity of $\pind{\ap} \mdn$ (part \ref{it:functrex}).  This proves part \ref{it:functsimp}.
\end{proof}

Unsurprisingly, parabolic subalgebras turn out to be important when classifying simple weight modules.  For this, the following result is germane.
\begin{theorem}[Fernando \protect{\cite[Thm.~4.18]{FerLie90}}] \label{thm:fernando}
Every simple weight $\ag$-module with \fdim{} weight spaces is isomorphic to $\pind{\ap} \mdn$, for some parabolic subalgebra $\ap \subseteq \ag$, with Levi factor $\al$ of $\type{AC}$-type, and some simple dense $\al$-module $\mdn$.
\end{theorem}
\noindent We note that if the parabolic is a Borel (so $\al = \ah$), then all simple $\al$-modules are dense and parabolic induction results in \hw{} $\ag$-modules.  Of course, parabolic induction does nothing if $\al = \ag$.

Suppose now that the reductive subalgebra $\al \subseteq \ag$ is of \type{AC}-type.  Then, there exists a semisimple coherent family $\mdc$ for $\al$, by \cref{thm:mathieu1}\ref{it:existence}.  As a natural generalisation of coherent families, we offer the following definition.
\begin{definition} \label{def:par}
	A \emph{parabolic family} of $\ag$-modules is a module $\mdp$ isomorphic to $\pind{\ap} \mdc$, for some parabolic subalgebra $\ap \subseteq \ag$, whose Levi factor $\al$ is of $\type{AC}$-type, and some coherent family $\mdc$ of $\al$-modules.
\end{definition}
\noindent Obviously, a coherent family is just a parabolic family corresponding to the parabolic subalgebra $\ag$.  Note that since $\mdp \cong \pind{\ap} \mdc$ by definition, we always have $\uinv{\ap} \mdp \cong \mdc$, by \cref{prop:functors}\ref{it:functinv}.

We remark that we were tempted to instead coin the term ``parabolic coherent family'' for the $\ag$-modules of \cref{def:par}.  However, these families are not necessarily ``coherent'' in the sense of Mathieu because their weight multiplicities need not be constant, even if we restrict to weights that differ by elements of the weight space $\ahs^*$ of $\al$.  Nevertheless, Fernando's theorem suggests that we will be able to use parabolic families to classify weight modules.

As for coherent families, a parabolic family
\begin{equation} \label{eq:decompP}
	\mdp = \bigoplus_{\lambda \in \ah^*/\rlat_{\ag}} \mdp_{\lambda}
\end{equation}
is said to be \emph{irreducible}, if at least one of the $\mdp_\lambda$ is a simple $\ag$-module, and \emph{semisimple}, if all of the $\mdp_{\lambda}$ are semisimple.  By \cref{prop:functors}, these notions are equivalent to $\mdc$ being irreducible and semisimple, respectively.

It is convenient at this point to note two useful facts.
\begin{proposition} \label{prop:finitelength}
	\leavevmode
	\begin{enumerate}
		\item \label{it:fincoh} \cite[Lem.~3.3]{MatCla00} The direct summands $\mdc_{\lambda}$ of a coherent family $\mdc$ of $\al$-modules have finite length.
		\item \label{it:finpar} \cite[Thm.~4.21]{FerLie90} Every finitely generated weight $\ag$-module with \fdim\ weight spaces has finite length.
	\end{enumerate}
\end{proposition}
\noindent From part~\ref{it:fincoh}, we learn that the $\mdc_{\lambda}$ are finitely generated, hence so are the $\mdp_{\lambda} \cong \pind{\ap} \mdc_{\lambda}$.  Consequently, the direct summands $\mdp_{\lambda}$ of a parabolic family $\mdp$ also have finite length, by part~\ref{it:finpar}.  It follows that any given parabolic family of $\ag$-modules has a \emph{semisimplification}, this being the semisimple parabolic family of $\ag$-modules obtained by replacing each of its direct summands by the direct sum of the summand's composition factors.  Clearly, the semisimplification of an irreducible parabolic family will also be irreducible.

If $\ap$ is a Borel and $\mdc$ is an irreducible semisimple coherent family for $\al = \ah$, then $\mdc$ is just a one-dimensional $\ah$-module.  The parabolic family $\mdp = \pind{\ap} \mdc$ is thus a simple \hwm{} (with respect to the Borel $\ap$).  When $\ap = \ag$, we instead get $\mdp = \mdc$.  This construction therefore interpolates between simple \hwms{} and coherent families for $\ag$.

We conclude this \lcnamecref{sec:para} with a few simple observations about the relationship between coherent and parabolic families.
\begin{proposition}\label{prop:CvsP}
	Let $\mdc$ be a coherent family of $\al$-modules and let $\mdp = \pind{\ap} \mdc$ be the associated parabolic family of $\ag$-modules.  Then:
	\begin{enumerate}
		\item \label{it:C=P} The $\al$-module embedding $\mdc \cong \wun \otimes \mdc \ira \mdp$ has the property that the weight spaces satisfy $\mdc(\mu) = \mdp(\mu)$, for all $\mu \in \supp{\mdc}$.
		\item \label{it:parpoly} The function $\tr_{\mdp(\mu)} U$ is polynomial in $\mu\in \supp{\mdc}$ for any $U\in\cent{\ah}{\ag}$.
		\item \label{it:simpquot} If $\mdm$ is a simple quotient of $\mdp$, then $\uinv{\ap} \mdm$ is a simple quotient of $\mdc$.
	\end{enumerate}
\end{proposition}
\begin{proof}
	For \ref{it:C=P}, first note that the \pbw{} theorem gives $\mdp(\mu)=\mdc(\mu)+\brac*{\au^- \envalg{\au^-} \otimes \mdc}(\mu)$.  Since $\mu \in \supp{\mdc} = \zeta + \ahs^*$ and the weights of $\au^- \envalg{\au^-}$ have empty intersection with $\ahs^*$, it follows that $\brac*{\au^- \envalg{\au^-} \otimes \mdc}(\mu)=0$.  This proves the first assertion.  The same intersection argument also shows that $\cent{\ah}{\ag}$ may be decomposed, again \emph{\`{a} la} \pbw, as $\cent{\ah}{\al} \oplus 	(\envalg{\ag} \au)^{\ah}$.  Since $\mdc$ is a coherent family for $\al$, the elements of $\cent{\ah}{\al}$ act polynomially on the $\mdc(\mu)$ with $\mu\in\supp{\mdc}$, while those of $(\envalg{\ag} \au)^{\ah}$ act as $0$.  Assertion \ref{it:parpoly} now follows from \ref{it:C=P}.

	To prove \ref{it:simpquot}, suppose that we have a simple quotient $\pi \colon \mdp \sra \mdm$.  Composing this with the inclusion $\mdc \ira \mdp$ from part \ref{it:C=P}, we get an $\al$-module homomorphism $v \in \mdc \mapsto \pi(\wun \otimes v) \in \mdm$ whose image is easily checked to lie in $\uinv{\ap} \mdm = \mdm^{\au}$.  If we assume that the image $\pi(\wun \otimes \mdc)$ is $0$, then we get
	\begin{equation}
		\mdm = \pi(\mdp) = \pi(\envalg{\ag} \cdot \wun \otimes \mdc) = \envalg{\ag} \cdot \pi(\wun \otimes \mdc) = 0,
	\end{equation}
	a contradiction.  We conclude that the composition $\mdc \to \uinv{\ap} \mdm$ is surjective as the target is a simple $\al$-module, by \cref{prop:functors}\ref{it:functlex}.
\end{proof}

\section{Simple module classification} \label{sec:class}

As before, let $\ag$ be a \fdim{} simple Lie algebra with fixed Cartan subalgebra $\ah$.  Our aim in this \lcnamecref{sec:class} is to classify the simple weight $\ag$-modules, with \fdim{} weight spaces, that are annihilated by some two-sided ideal $\aai$ of $\envalg{\ag}$.  To this end, we introduce
\begin{equation}
	\aaz = \frac{\envalg{\ag}}{\aai}
\end{equation}
and study the classification of simple weight $\aaz$-modules.  The motivating example corresponds to taking $\aai$ to be the Zhu ideal $\aai_{\kk}$ and $\aaz$ to be the Zhu algebra $\aaz_{\kk}$ of the simple level-$\kk$ affine vertex algebra associated to $\ag$, as in \cref{sec:intro}.  Another important example corresponds to taking $\aai$ to be the annihilating ideal of a given simple $\ag$-module.

We shall first determine when a given coherent family of $\ag$-modules is a $\aaz$-module before upgrading the result to parabolic families of $\ag$-modules.  This case serves to illustrate the strategy of the general proof with a minimum of complications.  Recall that $\cent{\ah}{\ag}$ denotes the centraliser of $\ah$ in $\envalg{\ag}$.  Let
\begin{equation}
	\aaa = \aai \cap \cent{\ah}{\ag}
\end{equation}
and note that $\aaa$ is a two-sided ideal of $\cent{\ah}{\ag}$.  We commence with the following very useful \lcnamecref{lem:Zmod}, whose underlying idea is surely well known (see \cite{AdaSom94} for example).
\begin{lemma} \label{lem:Zmod}
	A simple weight $\ag$-module $\mdm$ is a $\aaz$-module if and only if $\aaa \cdot v = 0$ for some non-zero $v \in \mdm$.
\end{lemma}
\begin{proof}
	If $\mdm$ is a $\aaz$-module, then $\aai \cdot \mdm = 0$, hence $\aaa \cdot \mdm = 0$ as required.  Suppose therefore that $\aaa \cdot v = 0$ for some non-zero weight vector $v \in \mdm$.  We may decompose $\aai$ as $\aaa \oplus \aab$, where the elements of $\aab$ have non-zero weights.  As $v$ is cyclic and $\aai$ is a right-ideal of $\envalg{\ag}$, we have
	\begin{equation}
		\aai \cdot \mdm = \aai \cdot \envalg{\ag} \cdot v = \aai \cdot v = \aaa \cdot v \oplus \aab \cdot v = \aab \cdot v.
	\end{equation}
	However, the (non-zero) elements of $\aab$ have non-zero weights, so it follows that $v \notin \aab \cdot v$.  This proves that $\aai \cdot \mdm$ is a proper submodule of $\mdm$, hence it is $0$ because $\mdm$ is simple.
\end{proof}

Consider now a semisimple coherent family $\mdc$ for $\ag$.  We choose a simple bounded submodule $\mdh \subset \mdc$ (this exists by \cref{thm:mathieu1}\ref{it:hwsubmod}).  We shall suppose that $\mdh$ is a $\aaz$-module, so that $\aai \cdot \mdh = 0$.  Then, $\tr_{\mdh(\mu)} a = 0$ for all $a \in \aai$ and $\mu \in \supp{\mdh}$.  But, $\mdh(\mu) = \mdc(\mu)$ for all $\mu \in \suppess{\mdh}$ (\cref{thm:mathieu1}\ref{it:degrees}), a set that is Zariski dense in $\ah^*$ (\cref{thm:mathieu1}\ref{it:zariski}).  We therefore have
\begin{equation} \label{eq:zerotrace}
	\tr_{\mdc(\mu)} a = 0, \qquad \text{for all}\ a \in \aaa\ \text{and}\ \mu \in \supp{\mdc} = \ah^*,
\end{equation}
since the trace of the action of $a \in \aaa \subset \aai$ is polynomial in $\mu$.  Now, $\dim \mdc(\mu) < \infty$, so replacing $a$ in \eqref{eq:zerotrace} by $a^n$, $n \in \ZZ_{>0}$, shows that $0$ is the only eigenvalue of $a$.  We conclude that every $a \in \aaa$ acts nilpotently on every $\mdc(\mu)$, $\mu \in \ah^*$.

Choose a simple direct summand $\mdm \subset \mdc$.  Then, each non-zero weight space $\mdm(\mu)$ is a simple $\cent{\ah}{\ag}$-module.  It follows that $\aaa \cdot \mdm(\mu)$ is either $0$ or $\mdm(\mu)$ because $\aaa$ is an ideal in $\cent{\ah}{\ag}$.  But, $\aaa \cdot \mdm(\mu) = \mdm(\mu) \neq 0$ would imply that any non-zero $v$ generates $\mdm(\mu)$, as an $\aaa$-module, and so satisfies $v \in \aaa \cdot v$.  However, having $v = av$, for some $a \in \aaa$, contradicts our earlier conclusion that $a$ must act nilpotently on $\mdm(\mu) \subseteq \mdc(\mu)$.  We therefore conclude that $\aaa \cdot \mdm(\mu) = 0$ for all $\mu \in \ah^*$, hence that $\aaa$ annihilates $\mdm$.  By \cref{lem:Zmod}, every simple direct summand $\mdm \subset \mdc$ is a $\aaz$-module, hence so is $\mdc$, as desired.

Note that every simple bounded \hw{} $\aaz$-module $\mdh$ embeds into some irreducible semisimple coherent family $\mdc$ (\cref{thm:mathieu1}\ref{it:cohextension}).  By the above argument, $\mdc$ and all its direct summands are $\aaz$-modules.  By choosing $\mdh$ above to be \hw{}, we see that the classification of semisimple coherent families that are $\aaz$-modules is therefore essentially equivalent to that of simple bounded \hw{} $\aaz$-modules.
\begin{proposition} \label{thm:cohfamclass}
	An irreducible semisimple coherent family for $\ag$ is a $\aaz$-module if and only if any (and thus all) of its bounded \hw{} submodules are.
\end{proposition}
\noindent Clearly, every \infdim{} \hw{} submodule of a coherent family is bounded.

We now extend this to a classification of \emph{all} simple weight $\aaz$-modules, with \fdim{} weight spaces, in terms of the classification of highest-weight $\aaz$-modules.  Recall the restriction- and induction-type functors $\uinv{\ap}$ and $\pind{\ap}$ from \cref{sec:para}.
\begin{definition} \label{def:lbound}
	Given a parabolic subalgebra $\ap \subseteq \ag$ with Levi factor $\al$, we say that a $\ag$-module $\mdm$ is \emph{$\al$-bounded} if $\uinv{\ap} \mdm$ is a bounded $\al$-module.
\end{definition}
\begin{proposition} \label{prop:parfamclass}
	Given a choice of parabolic subalgebra $\ap \subseteq \ag$, with non-abelian Levi factor $\al$ of \type{AC}-type, an irreducible semisimple parabolic family for $\ag$ will be a $\aaz$-module if and only if any (and thus all) of its $\al$-bounded \hw{} submodules are.
\end{proposition}
\begin{proof}
	Let $\mdp$ be such an irreducible semisimple parabolic family and let $\mdc = \uinv{\ap} \mdp$, so that $\mdc$ is a coherent family of $\al$-modules with $\mdp \cong \pind{\ap} \mdc$.  Suppose that $\mdh \subset \mdp$ is a simple $\al$-bounded submodule that happens to be a $\aaz$-module: $a \cdot \mdh(\mu) = 0$ for all $a \in \aai$ and $\mu \in \supp{\mdh}$.  We now focus on the $\al$-submodule $\uinv{\ap} \mdh$ of $\mdh$, restricting $a$ to $\aaa$ and $\mu$ to $\suppess{\uinv{\ap} \mdh}$.  As $\uinv{\ap} \mdh$ is a simple bounded $\al$-submodule of $\mdc$, by \cref{prop:functors}\ref{it:functlex}, its essential support is Zariski-dense in $\supp{\mdc}$.  Moreover, $\mu \mapsto \tr_{\mdp(\mu)} a$ is polynomial in $\mu \in \supp{\mdc}$, for each $a \in \aaa$, by \cref{prop:CvsP}\ref{it:parpoly}.  We therefore conclude, as in the coherent family argument above, that $\aaa$ acts nilpotently on each $\mdp(\mu) = \mdc(\mu)$ with $\mu \in \supp{\mdc}$.

	Any simple $\ag$-submodule $\mdm \subset \mdp$ has a non-zero image under $\uinv{\ap}$ because a zero image would mean that $\mdm$ has zero intersection with $\uinv{\ap} \mdp \cong \mdc$ and hence be zero in $\mdp \cong \pind{\ap} \mdc$.  We may therefore choose a (non-zero) weight vector $v \in \uinv{\ap} \mdm$ and let $\mu$ denote its weight.  Since $\uinv{\ap} \mdm \subset \mdc$, by \cref{prop:functors}\ref{it:functlex}, it follows that $\mu \in \supp{\mdc}$ and so $\aaa$ acts nilpotently on $v \in (\uinv{\ap} \mdm)(\mu)$.  As above, $v$ generating the simple $\cent{\ah}{\ag}$-module $(\uinv{\ap} \mdm)(\mu)$ under the action of $\aaa$ contradicts the nilpotence of this action.  $\aaa$ must therefore annihilate $v \in \mdm$, whence $\mdm$ must be a $\aaz$-module, by \cref{lem:Zmod}, and the proof is complete.
\end{proof}

We are now ready to prove our classification result.
\begin{theorem}\label{thm:classification}
	Let $\ag$ be a \fdim{} simple Lie algebra and let $\aaz$ be a quotient of $\envalg{\ag}$ by a two-sided ideal.  Then, a simple weight $\ag$-module $\mdm$, with \fdim{} weight spaces, is a $\aaz$-module if and only if either of the following statements hold:
	\begin{itemize}
		\item $\mdm$ is a \hw{} $\aaz$-module, with respect to some Borel subalgebra of $\ag$.
		\item There is a parabolic subalgebra $\ap \subseteq \ag$, with non-abelian Levi factor $\al$ of \type{AC}-type, and a corresponding irreducible semisimple parabolic family $\mdp$ of $\ag$-modules such that $\mdm$ is isomorphic to a submodule of $\mdp$ and some submodule of $\mdp$ is an $\al$-bounded \hw{} $\aaz$-module.
	\end{itemize}
\end{theorem}
\begin{proof}
	\cref{prop:parfamclass} shows that every submodule $\mdm$ of such a parabolic family $\mdp$ is a $\aaz$-module.  Conversely, let $\mdm$ be a simple weight $\aaz$-module, with \fdim{} weight spaces.  We assume that $\mdm$ is not \hw{}, with respect to any Borel.  Then, \cref{thm:fernando} says that $\mdm \cong \pind{\ap} \mdn$, for some parabolic subalgebra $\ap \subseteq \ag$, with non-abelian Levi factor $\al$ of \type{AC}-type, and some simple dense $\al$-module $\mdn$.  As simple dense modules over a non-abelian $\al$ are bounded (\cref{sec:coh}), $\mdn$ embeds in an irreducible semisimple coherent family $\mdc$ of $\al$-modules, by \cref{thm:mathieu1}\ref{it:cohextension}.  But, \cref{thm:mathieu1}\ref{it:hwsubmod} ensures that $\mdc$ contains a simple bounded \hw{} submodule $\mdh$.  It thus follows from \cref{prop:functors}\ref{it:functsimp} that the irreducible semisimple parabolic family $\mdp = \pind{\ap} \mdc$ contains the simple $\al$-bounded \hw{} $\ag$-module $\pind{\ap} \mdh$.  It only remains to show that $\pind{\ap} \mdh$ is a $\aaz$-module.  However, this follows from \cref{prop:parfamclass} and the fact that $\mdm$ is a simple $\al$-bounded $\aaz$-module.
\end{proof}

This \lcnamecref{thm:classification} reduces the classification of simple weight $\aaz$-modules to that of simple highest-weight $\aaz$-modules and parabolic subalgebras with non-abelian Levi factors of \type{AC}-type.  The former is a difficult problem in general, through tractable in many important cases, while the latter is essentially combinatorial.  Note that \cref{thm:zhuclass} is a straightforward corollary of \cref{thm:classification} with $\aai$ taken to be the Zhu ideal of the simple level-$\kk$ affine vertex algebra $\affvoa{\kk}{\ag}$.

\section{Indecomposable modules} \label{sec:indec}

In this section, we study irreducible, but non-semisimple, parabolic families in order to determine when certain indecomposable $\ag$-modules are $\aaz$-modules.  Let $\roots{\al}$ denote the root system of $\al$ and let $e^{\alpha}$ denote the root vector corresponding to the root $\alpha \in \roots{\al}$.
\begin{definition} \label{def:bijective}
	A weight module $\mdn$ over $\al$ is \emph{$\alpha$-bijective}, for some given $\alpha \in \Delta_{\al}$, if $e^\alpha$ acts bijectively on $\mdn$.
\end{definition}
\noindent
Many examples of such modules were constructed by Mathieu \cite[Lem.~4.5]{MatCla00} using a powerful tool called twisted localisation.  In particular, for any irreducible semisimple coherent family $\mdc'$, there are $\alpha$-bijective coherent families $\mdc$ such that $\mdc'$ is the semisimplification of $\mdc$.  For $\ag=\SLA{sl}{2}$, the dense direct summands $\mdc_{\lambda}$ of an $\alpha$-bijective coherent family may also be constructed explicitly, see \cite[Sec.~7.8.16]{DixEnv96} or \cite[Sec.~3.3]{MazLec10}, or as modules induced from one-dimensional modules of the centraliser $\cent{\ah}{\SLA{sl}{2}}$, see \cite[Ex.~3.99]{MazLec10} or \cite[Sec.~3.2]{KawRel18}.  We note that this induction procedure can also result in indecomposable dense modules that are not $\alpha$-bijective for any root $\alpha$.

Given a simple bounded $\al$-module $\mdn$, let $\injmonoid{\mdn}$ denote the additive monoid generated by the roots $\alpha \in \Delta$ whose root vectors $e^\alpha$ act injectively on $\mdn$.  We need two straightforward results about these monoids.
\begin{proposition}[Mathieu \cite{MatCla00}] \label{thm:mathieu2}
	Let $\mdn$ be a simple bounded $\al$-module.  Then:
	\begin{enumerate}
		\item \label{it:mongp} [Lem.~3.1] The group-completion of the monoid $\injmonoid{\mdn}$ is $\rlat_{\al}$ (the root lattice of $\al$).
		\item \label{it:monsuppess} [Prop.~3.5i] For any $\lam\in\suppess{\mdn}$, we have $\lam+\injmonoid{\mdn} \subseteq \suppess{\mdn}$.
	\end{enumerate}
\end{proposition}

Our goal is to prove the following \lcnamecref{thm:parabolicindecomposable}.
\begin{theorem}\label{thm:parabolicindecomposable}
	Let $\ap$ be a parabolic subalgebra of $\ag$ with non-abelian Levi factor $\al$ of \type{AC}-type.
	Let $\mdc$ be an irreducible $\alpha$-bijective coherent family of $\al$-modules, for some $\alpha\in\Delta_\al$.
	Let $\mdp$ denote the irreducible parabolic family of $\ag$-modules induced from $\mdc$.
	Suppose that an $\al$-bounded \hw{} submodule $\mdh$ of $\mdp$ is a $\aaz$-module.
	Then $\mdp$, and hence all its subquotients, are also $\aaz$-modules.
\end{theorem}
\noindent The idea behind the proof is that if such a parabolic family has an $\al$-bounded \hw{} $\aaz$-module as a submodule, then all its simple quotients are also $\aaz$-modules, by \cref{thm:classification}, hence the ideal $\aai$ must map each direct summand of the parabolic family into its radical.  We will show that the weight multiplicities of each simple quotient are frequently equal to those of the corresponding direct summand so that those of the radical are frequently zero.  This happens sufficiently often to prove that $\aai$ in fact maps each direct summand of the parabolic family to zero.

There are technical details required to make this idea precise.  For these, we have the following four \lcnamecrefs{lem:Ugnoeth}.  The first follows immediately from the well-known fact, see \cite[Cor.~2.3.8]{DixEnv96} for example, that $\envalg{\ag}$ is noetherian.
\begin{lemma} \label{lem:Ugnoeth}
	The ideal $\aai$ is finitely generated as a left-ideal of $\envalg{\ag}$.
\end{lemma}

\begin{lemma}\label{lem:head}
	Let $\mdc$ be an $\alpha$-bijective coherent family of $\al$-modules and $\mdp$ be the induced irreducible parabolic family of $\ag$-modules, as in \cref{thm:parabolicindecomposable}.  Then every simple quotient of $\mdp$ is $\al$-bounded.
\end{lemma}
\begin{proof}
	Let $\mdm$ be a simple quotient of $\mdp$.  By \cref{prop:CvsP}\ref{it:simpquot}, $\uinv{\ap} \mdm$ is a simple quotient of $\mdc$, so its multiplicities are uniformly bounded.  Let $\eta \colon \mdc \to \uinv{\ap} \mdm$ denote the quotient map.  The \lcnamecref{lem:head} will therefore follow if we can show that $\uinv{\ap} \mdm$ is \infdim{}.  Assume the contrary: that $\dim \uinv{\ap} \mdm < \infty$.  Then, there exists $\mu \in \supp{\uinv{\ap} \mdm}$ such that $\mu - \alpha \notin \supp{\uinv{\ap} \mdm}$.  As $e^{\alpha}$ acts bijectively on $\mdc$, we obtain
	\begin{equation}
		(\uinv{\ap} \mdm)(\mu) = \eta(\mdc(\mu)) = \eta(e^{\alpha} \cdot \mdc(\mu-\alpha)) = e^{\alpha} \cdot \eta(\mdc(\mu-\alpha)) = e^{\alpha} \cdot (\uinv{\ap} \mdm)(\mu-\alpha) = 0,
	\end{equation}
	a contradiction.  It follows that $\uinv{\ap} \mdm$ must be \infdim{}, completing the proof.
\end{proof}

\begin{lemma}\label{lem:fat}
	Let $\mdn$ be a simple bounded $\al$-module.  Then, for any finite subset $S$ of $\rlat_{\al}$, there is a weight $\mu \in \suppess{\mdn}$ such that $\mu+S\subset \suppess{\mdn}$.
\end{lemma}
\begin{proof}
	Note first that $\suppess{\mdn}$ is not empty (\cref{thm:mathieu1}\ref{it:zariski}).  So, choose $\lam\in\suppess{\mdn}$ and let $S=\set{\beta_1,\ldots,\beta_n} \subset \rlat_{\al}$.  Since the monoid $\injmonoid{\mdn}$ generates $\rlat_{\al}$ (\cref{thm:mathieu2}\ref{it:mongp}), the elements of $S$ have the form $\beta_i=\mu_i-\nu_i$, with $\mu_i,\nu_i\in \injmonoid{\mdn}$.  Set $\nu=\nu_1+\cdots+\nu_n$ and $\mu = \lambda + \nu$.  Then, $\nu \in \injmonoid{\mdn}$ so $\mu \in \suppess{\mdn}$ (\cref{thm:mathieu2}\ref{it:monsuppess}).  Moreover, $\nu+S \subset \injmonoid{\mdn}$ and so $\mu+S \subset \suppess{\mdn}$ (\cref{thm:mathieu2}\ref{it:monsuppess} again).  This proves the assertion.
\end{proof}

Recall that $\umonoid$ denotes the monoid generated by the roots of $\au$; it satisfies $\umonoid \cap \rlat_{\al} = 0$.  The monoid generated by the roots of $\au^-$ is therefore $-\umonoid$.
\begin{lemma} \label{lem:intersection}
	If $v$ is a non-zero weight vector in $\mdp$, then $\brac*{\envalg{\au} \cdot v} \cap (\wun \otimes \mdc)$ is non-zero.
\end{lemma}
\begin{proof}
	Suppose that $\envalg{\au} \cdot v$ has zero intersection with $\wun \otimes \mdc$.  Acting with $\envalg{\al}$ will not change this because it just adds elements of $\rlat_{\al}$ to the weights and the weights of $\mdc$ are already a shift of $\rlat_{\al}$.  Moreover, acting with $\envalg{\au^-}$ will not change this intersection either because $-\umonoid \cap \rlat_{\al} = 0$.  It now follows from \pbw{} that the $\envalg{\ag}$-submodule generated by $v$ has zero intersection with $\wun \otimes \mdc$ and is therefore $0$, by definition of $\mdp = \pind{\ap} \mdc$.
\end{proof}

We now prove \cref{thm:parabolicindecomposable}.
\begin{proof}[Proof of \cref{thm:parabolicindecomposable}]
	Recall the decomposition \eqref{eq:decompP} of $\mdp$ into $\ag$-submodules $\mdp_\lam$, $\lam \in \ah^*/\rlat_{\ag}$, which need not be simple.  We will show that each of the $\mdp_\lam$ are annihilated by $\aai$.  To do this, fix $\lam \in \ah^*/\rlat_{\ag}$ and let $\mdm_i$, $i \in I$, denote the simple quotients of $\mdp_\lam$.  We recall that the $\mdp_\lam$ are finite-length (\cref{prop:finitelength}), hence $I$ is non-empty.  Since each $\mdm_i$ is isomorphic to a submodule of the semisimplification of $\mdp$, which contains an $\al$-bounded \hw{} $\aaz$-module by hypothesis, it follows from \cref{thm:classification} that $\mdm_i$ is likewise a $\aaz$-module.  Thus, $\aai \cdot \mdm_i = 0$ for all $i \in I$, hence $\aai \cdot \mdp_\lam\subseteq \rad \mdp_\lam$, the radical of $\mdp_\lam$ (the intersection of its maximal proper submodules).

	Suppose now that $\mu \in \suppess{\uinv{\ap} \mdm_i}$.  Then, we have
	\begin{equation}
		\dim \mdm_i(\mu) \ge \dim (\uinv{\ap} \mdm_i)(\mu) = \dim \mdc_{\lambda}(\mu) = \dim \mdp_{\lambda}(\mu),
	\end{equation}
	by \cref{thm:mathieu1}\ref{it:degrees} and \cref{prop:CvsP}\ref{it:C=P}, which establishes the equality $\dim \mdm_i(\mu) = \dim \mdp_{\lambda}(\mu)$.  We conclude that $(\rad \mdp_{\lambda})(\mu) = 0$ whenever $\mu \in \suppess{\uinv{\ap} \mdm_i}$ for some $i \in I$.

	\cref{lem:Ugnoeth} ensures that there exist a finite number of elements $U_1,\dots,U_n$ that generate $\aai$ as a left-ideal of $\envalg{\ag}$.  Without loss of generality, we may take these elements to be weight vectors of $\envalg{\ag}$, denoting their weights by $\nu_j$, $j=1,\dots,n$.  For each $j$, define a set $S_j \subset \ah^*$ by
	\begin{equation}
		S_j = \brac*{\nu_j + \umonoid} \cap \rlat_{\al}
	\end{equation}
	and note that each $S_j$ is finite, a fact that is easily established by expanding $\nu_j$ in a basis of $\ah^*$ consisting of roots of $\al$ and $\au$.  Let $S$ denote the union of the $S_j$.  As each $\mdm_i$ is simple and $\al$-bounded, by \cref{lem:head}, it now follows from \cref{lem:fat} (with $\mdn = \uinv{\ap} \mdm_i$) that there exists a weight $\mu_i \in \suppess{\uinv{\ap} \mdm_i}$, for each $i \in I$, such that $\mu_i+S \subset \suppess{\uinv{\ap} \mdm_i}$.

	Recall that $U_j \cdot \mdp_{\lambda}(\mu_i) \subseteq (\rad \mdp_{\lambda})(\mu_i + \nu_j)$.  We want to show that the \rhs{}, and thus the \lhs{}, of this inclusion is zero.  To do so, act with $\envalg{\au}$ in order to bring the weight $\mu_i + \nu_j$ back to an element of $\mu_i + \rlat_{\al}$.  In other words, consider the subspace $\brac*{\envalg{\au} \cdot U_j \cdot \mdp_{\lambda}(\mu_i)} \cap (\wun \otimes \mdc_{\lambda})$ corresponding to weights lying in $\brac*{\mu_i + \nu_j + \umonoid} \cap (\mu_i + \rlat_{\al}) = \mu_i + S_j \subset \suppess{\uinv{\ap} \mdm_i}$.  We conclude that
	\begin{equation}
		\brac*{\envalg{\au} \cdot U_j \cdot \mdp_{\lambda}(\mu_i)} \cap (\wun \otimes \mdc_{\lambda}) \subseteq (\rad \mdp_{\lambda})(\mu_i + S_j) = 0,
	\end{equation}
	because the radical vanishes for weights in the essential support of some $\uinv{\ap} \mdm_i$.  It now follows from \cref{lem:intersection} that $U_j \cdot \mdp_{\lambda}(\mu_i) = 0$, for each $j=1,\dots,n$, as desired.

	Finally, the $U_j$ generate $\aai$ as a left-ideal, so $\aai \cdot \mdp_\lam(\mu_i)=0$ for all $i \in I$.  As $\aai$ is also a right-ideal of $\envalg{\ag}$, it therefore annihilates the submodule of $\mdp_\lam$ generated by each $\mdp_\lam(\mu_i)$, $i \in I$.  It therefore annihilates the sum of these submodules, which is clearly $\mdp_\lam$.  $\mdp_\lam$ is thus a $\aaz$-module, for all $\lam \in \ah^* / \rlat_{\ag}$, hence so is $\mdp$.
\end{proof}

Our \cref{thm:zhuind} is obtained by applying the induction functor of Zhu and Li to the parabolic family of $\aaz$-modules guaranteed by \cref{thm:parabolicindecomposable} (see \cref{thm:zhu}\ref{it:zhuindec}).  In this application, $\aaz$ is taken to be the Zhu algebra of an affine \voa{} $\vkg$ (as in \cref{sec:results}).

It is natural to consider a direct proof of \cref{thm:classification} using the twisted localisation functors introduced by Mathieu in \cite{MatCla00} and we hope to come back to this point in the future.  This leads us to ask if \cref{thm:parabolicindecomposable} can likewise be proved using localisation.  This is unclear to us at present because it is not obvious that every $\alpha$-bijective parabolic family can be constructed in this fashion.

\section{Coherent families of simple Lie algebras of \type{AC}-type}\label{sec:ac}

In this \lcnamecref{sec:ac}, we recall Mathieu's explicit classification \cite{MatCla00} of irreducible semisimple coherent families over a \fdim{} simple Lie algebra $\ag$.  As mentioned above, there are no coherent families if $\ag$ is not of type \type{A} or \type{C}.  For this classification, it will be convenient to introduce some terminology for weights $\lam \in \ah^*$.  In particular, we say that $\lam$ is \emph{integral}, \emph{shifted-singular} or \emph{shifted-regular} if $\lam + \wvec$ belongs to the weight lattice $\wlat_{\ag}$, lies on a Weyl chamber wall, or lies in the interior of a Weyl chamber, respectively.

Choose a Borel subalgebra of $\ag$, hence a notion of being \hw{}.  Let $\bounded_{\ag}$ denote the set of weights $\lambda \in \ah^*$ such that the simple \hw{} $\ag$-module of highest weight $\lambda$ is bounded.  As semisimple coherent families are invariant under the action of the Weyl group \cite[Prop.~6.2]{MatCla00}, it does not matter which choice of Borel we make.  Note that the set $\bounded_{\ag}$ is empty if $\ag$ is not of type \type{A} or \type{C} (\cref{thm:mathieu1}\ref{it:cohextension} and \ref{it:existence}).

\subsection{Type \type{A}}\label{sec:typea}

Let $\ag=\SLA{sl}{n+1}$ with $n\ge 1$.  With respect to our chosen Borel, we have simple roots $\sroot{1},\dots,\sroot{n}$, highest root $\hroot$, Weyl vector $\wvec$ and dominant integral weights $\wlat_{\ag}^{\ge}$.  For $\lam\in\ah^*$, set
\begin{equation}
	A(\lam)=\set*{i\in\set{1,\dots,n} \st \bilin{\lam+\wvec}{\scroot{i}} \notin\ZZ_{>0}},
\end{equation}
where the Killing form is normalised so that $\bilin{\hroot}{\hroot} = 2$.  Note that $A(\lambda) = \varnothing$ if and only if $\lambda \in \wlat_{\ag}^{\ge}$.

\begin{proposition}[Mathieu \protect{\cite[Lem.~8.1 and Prop.~8.5]{MatCla00}}] \label{prop:bounded-a}
	For $\ag = \SLA{sl}{n+1}$, the set $\bounded=\bounded_\ag$ consists of the elements $\lam\in\ah^*$ that satisfy at least one of the following conditions:
	\begin{enumerate}[ref=\textup{(\alph*)}]
		\item \label{it:Aedge} $A(\lam)=\set{1}$ or $\set{n}$.
		\item \label{it:Amid} $A(\lam)=\set{i}$ with $1<i<n$ and either $(\lam+\wvec,\scroot{i}+\scroot{i-1})\in\ZZ_{>0}$ or $(\lam+\wvec,\scroot{i}+\scroot{i+1})\in\ZZ_{>0}$.
		\item \label{it:Adoub} $A(\lam)=\set{i,i+1}$ with $1\le i<n$ and $(\lam+\wvec,\scroot{i}+\scroot{i+1})\in\ZZ_{>0}$.
	\end{enumerate}
\end{proposition}

For example, only \ref{it:Aedge} applies when $\ag = \SLA{sl}{2}$, hence $\bounded_{\SLA{sl}{2}} = \ah^* \setminus \wlat_{\SLA{sl}{2}}^{\ge}$ is the set of weights whose Dynkin label is not a non-negative integer.  For $\ag = \SLA{sl}{3}$, \ref{it:Amid} does not apply and $\bounded_{\SLA{sl}{3}}$ is the union of two sets: one consisting of the weights that have precisely one non-negative integer Dynkin label and the other consisting of the weights with no non-negative integer Dynkin labels but for which the sum of the Dynkin labels lies in $\ZZ_{\ge -1}$.

For $\lam,\mu\in \bounded$, we write $\lam \ra \mu$ if there exists $i \in A(\lam)$ with $\mu=\swref{i} \cdot \lam$.  Here, $\swref{i}$ is the simple reflection $\wref{\sroot{i}}$ of the Weyl group $\wgrp \cong \grp{S}_{n+1}$ of $\ag$ and $\cdot$ denotes the shifted action of $\wgrp$ on $\ah^*$: $w \cdot \lam = w(\lam+\wvec)-\wvec$.  For $\ag = \SLA{sl}{2}$, we thus have $\lambda \ra \swref{1} \cdot \lambda$ whenever the Dynkin label $\lambda_1$ of $\lambda \in \bounded_{\SLA{sl}{2}}$ is not in $\ZZ_{\le-2}$; however when $\lambda_1 \in \ZZ_{\le-2}$, there is no $\mu \in \bounded_{\SLA{sl}{2}}$ satisfying $\lambda \ra \mu$.  The relation $\ra$ is obviously not an equivalence relation, but it defines on $\bounded$ the structure of a directed graph: the vertices are the weights of $\bounded$ and we have an edge from $\lambda$ to $\mu$ if and only if $\lambda \ra \mu$.  We shall denote the set of connected components of this graph by $\bounded\big/(\ra)$.

We have the following classification result.
\begin{proposition}[Mathieu \protect{\cite[Thm.~8.6]{MatCla00}}] \label{prop:bij-a}
	There is a bijective correspondence between the set of (equivalence classes of) irreducible semisimple coherent families of $\SLA{sl}{n+1}$-modules and the set $\bounded\big/(\ra)$ of connected components of $\bounded$.  This correspondence sends an irreducible semisimple coherent family $\mdc$ to the set
	\begin{equation}
		\set*{\lam\in\ah^* \st \lam\not\in \wlat_{\ag}^{\ge}\ \text{and}\ \mdl_\lam\subset \mdc} \in \bounded\big/(\ra)
	\end{equation}
	of highest weights of \infdim{} highest-weight submodules of $\mdc$.
\end{proposition}

This shows that irreducible semisimple coherent families of $\SLA{sl}{n+1}$-modules are completely characterised by their bounded \hw{} submodules (and in fact, a single representative will do).  Because all elements of $\cent{\ah}{\ag}$ act polynomially on a given coherent family $\mdc$, it follows that each element of the centre $\cent{\ag}{\ag}$, that is each Casimir operator, acts as a constant on $\mdc$.  In other words, $\mdc$ has a definite central character.  It is therefore natural to ask whether the central character also completely characterises an irreducible semisimple coherent family.  The answer is interesting: ``usually, but not always''.

We recall that two \hwms{} have the same central character if and only if their highest weights are related by the shifted action of $\wgrp$.  Given $\lam \in \bounded$, the question asked above amounts to deciding whether $(\wgrp \cdot \lam) \cap \bounded$ is a single connected component in $\bounded$ or not.
\begin{proposition}[Mathieu \protect{\cite[Lem.~8.3]{MatCla00}}] \label{prop:connect-a}
	\leavevmode
	\begin{enumerate}
		\item If $\lam\in\bounded$ is integral, then the connected component $[\lam]\in\bounded\big/(\ra)$ has $n$ elements.  Otherwise, $[\lam]$ has $n+1$ elements. \label{it:conncomp-a}
		\item The intersection $(\wgrp \cdot \lam) \cap \bounded$ is a single connected component in $\bounded$ \emph{unless} $\lam$ is shifted-regular and integral, in which case it is the union of $n$ connected components.
	\end{enumerate}
\end{proposition}
\noindent We conclude that an irreducible semisimple coherent family of $\SLA{sl}{n+1}$-modules is completely characterised by its central character unless its  \hw{} submodules have shifted-regular integral highest weights (and if one does, then they all do).

We illustrate these ideas for $\ag=\SLA{sl}{2}$.  In this case, the connected components of $\bounded_{\SLA{sl}{2}} = \ah^* \setminus \wlat_{\SLA{sl}{2}}^{\ge}$ have the form $[\lam] = \set{\lambda}$, if $\lam$ is integral, and $[\lam] = \set{\lam,\swref{1} \cdot \lam}$ otherwise.  The set of connected components of $\bounded_{\SLA{sl}{2}}$ thus decomposes into shifted-regular integral, shifted-singular integral, and non-integral weights as follows:
\begin{equation}
	\bounded_{\SLA{sl}{2}} \big/ (\ra) = \bigcup_{\lam \in \ZZ_{\le-2}} \set{\lam \fwt{1}} \cup \set{-\fwt{1}} \cup \bigcup_{\lam \in \CC \setminus \ZZ} \set{\lam \fwt{1}, -(\lam+2) \fwt{1}}.
\end{equation}
Moreover, the central character always completely characterises the coherent families.  While there exist shifted-regular integral weights in $\bounded_{\SLA{sl}{2}}$ (those with $\lambda \in \ZZ_{\le-2}$), the (partial) $\wgrp$-orbits $(\wgrp \cdot \lam) \cap \bounded_{\SLA{sl}{2}} = \set{\lambda}$ coincide with the connected components in this case, consistent with \cref{prop:connect-a} (because $n=1$).

The case $\ag = \SLA{sl}{3}$ is more typical and we illustrate the set $\bounded_{\SLA{sl}{3}}$ in \cref{fig:sl3sp4} for convenience.  As before, $\bounded_{\SLA{sl}{3}}$ is partitioned into shifted-regular integral, shifted-singular integral, and non-integral weights.  It is easy to see that each non-integral weight $\lam \in \bounded_{\SLA{sl}{3}}$ gives rise to a length-$6$ (shifted) $\wgrp$-orbit whose intersection with $\bounded_{\SLA{sl}{3}}$ consists of three weights and represents one connected component.  The shifted-singular integral weights correspond to the intersections of the red and black lines.  This singularity means that the $\wgrp$-orbit's length is only $3$, but one element necessarily lies outside $\bounded_{\SLA{sl}{3}}$.  The remaining two weights again form a single connected component.  Finally, each shifted-regular integral weight yields a length-$6$ $\wgrp$-orbit whose intersection with $\bounded_{\SLA{sl}{3}}$ has four elements.  Because weights linked by $\ra$ must be related by a simple Weyl reflection, the intersection splits into two connected components of two elements each.

\begin{figure}
	\begin{tikzpicture}[>=stealth,scale=0.6]
		\draw[red,<->] (90:5.5) -- node[pos=0.03,above left] {\textcolor{black}{$\scriptstyle \lam_2$}} (-90:5.5);
		\draw[red,<->] (150:5.5) -- (-30:5.5);
		\draw[red,<->] (30:5.5) -- node[pos=0.01,right] {\textcolor{black}{$\scriptstyle \lam_1$}} (-150:5.5);
		\clip (-5,-5.15) rectangle (5,5.15);
		\foreach \i in {1,...,7}{
			\begin{scope}[on background layer]     
				\clip (-5,-5.15) rectangle (5,5.15); 
				\draw[very thick] ($(30:10) + (0,\i)$) -- ($(-150:10) + (0,\i)$);
				\draw[very thick] ($(150:10) + (0,\i)$) -- ($(-30:10) + (0,\i)$);
				\draw[very thick] ($(90:10) + (30:\i)$) -- ($(-90:10) + (30:\i)$);
			\end{scope}
			\foreach \j in {1,...,7}{
				\node[dwt] at ($(0,\i) + (30:\j)$) {};
				\node[rwt] at ($(0,\i) + (150:\j)$) {};
				\node[rwt] at ($(150:\i) - (30:\j)$) {};
				\node[rwt] at ($(30:\i) - (150:\j)$) {};
				\node[rwt] at ($(-30:\i) - (90:\j)$) {};
				                      }
			                      }
	\end{tikzpicture}
	\hspace{0.05\textwidth}
	\begin{tikzpicture}[>=stealth,scale=0.6]
		\draw[red,<->] (90:8.5) -- node[pos=0.03,above left] {\textcolor{black}{$\scriptstyle \lam_2$}} (-90:2.5);
		\draw[red,<->] (0:8.5) -- (180:2.5);
		\draw[red,<->] (135:3) -- (-45:3);
		\draw[red,<->] (45:10) -- node[pos=0.01,right] {\textcolor{black}{$\scriptstyle \lam_1$}} (-135:3);
		\clip (-2,-2) rectangle (8.25,8.25);
		\foreach \i in {1,...,7}{
			\foreach \j in {1,...,7}{
				\node[dwt] at ($(\i,\i+2*\j)$) {};
				\node[bwt] at ($(\i,\i+2*\j-1)$) {};
				\node[bwt] at ($(\i+2*\j-1,\i)$) {};
				                      }
			                      }
	\end{tikzpicture}
\caption{At left, the set $\bounded_{\SLA{sl}{3}}$, depicted as black lines, in (the real slice of) the weight space.  The white circles correspond to the dominant integral weights, which are not in $\bounded_{\SLA{sl}{3}}$, while the grey circles indicate the shifted-regular integral weights in $\bounded_{\SLA{sl}{3}}$.  The red lines correspond to the shifted Weyl chamber walls.  At right, the set $\bounded_{\SLA{sp}{4}}$, with the same conventions except that $\bounded_{\SLA{sp}{4}}$ is discrete and is therefore represented by the black circles.} \label{fig:sl3sp4}
\end{figure}
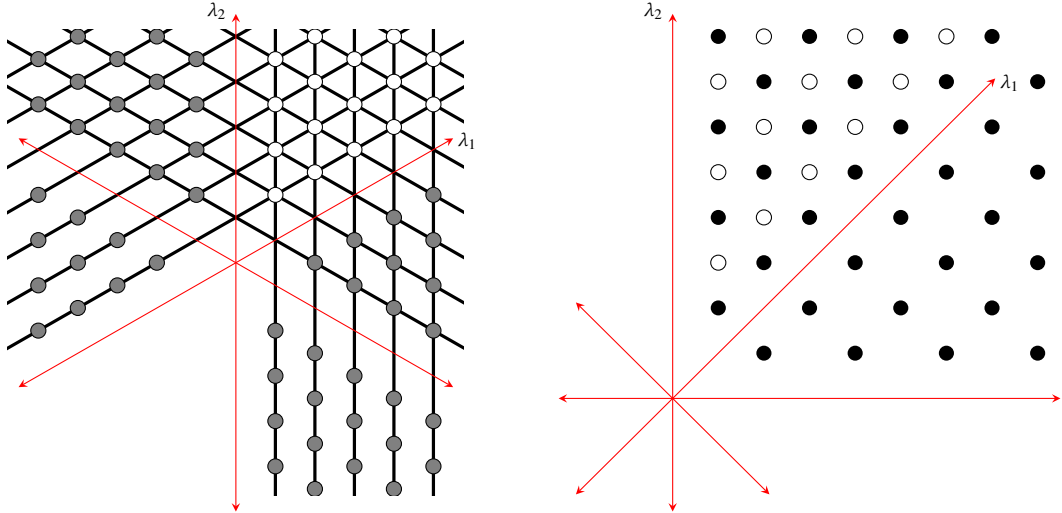

\subsection{Type \type{C}}\label{sec:typec}

The situation is somewhat more straightforward for $\ag=\SLA{sp}{2n}$ (with $n\ge 2$).  We fix an ordering of the simple roots in which consecutive roots are connected in the Dynkin diagram, $\sroot{1},\ldots,\sroot{n-1}$ are short and $\sroot{n}$ is long.

\begin{proposition}[Mathieu \protect{\cite[Lems.~9.1 and 9.2]{MatCla00}}] \label{prop:bounded-c}
	For $\ag = \SLA{sp}{2n}$, the set $\bounded=\bounded_\ag$ consists of the elements $\lam\in\ah^*$ which satisfy all of the following conditions:
	\begin{enumerate}
		\item $\bilin{\lam}{\scroot{i}} \in \ZZ_{\ge0}$ for any $i\neq n$.
		\item $\bilin{\lam}{\scroot{n}} \in \ZZ+\frac{1}{2}$.
		\item $\bilin{\lam}{\scroot{n-1}+2\scroot{n}} \in \ZZ_{\ge-2}$.
	\end{enumerate}
\end{proposition}
\noindent Note that $\bounded$ is clearly discrete in this case.  We illustrate $\bounded_{\SLA{sp}{4}}$ in \cref{fig:sl3sp4} (right).

For $\lam,\mu\in \bounded$, we write $\lam\ra \mu$ if $\mu = \swref{n} \cdot \lam$, where we recall that the Weyl group of $\SLA{sp}{2n}$ is $\wgrp \cong \grp{S}_n \ltimes \ZZ_2^n$.
\begin{proposition}[Mathieu \protect{\cite[Thm.~9.3]{MatCla00}}] \label{prop:connect-c}
	\leavevmode
	\begin{enumerate}
		\item There is a bijective correspondence between the set of (equivalence classes of) irreducible semisimple coherent families of $\SLA{sp}{2n}$-modules and the set $\bounded\big/(\ra)$ of connected components in $\bounded$.  This correspondence sends an irreducible semisimple coherent family $\mdc$ to the set
	\begin{equation}
		\set*{\lam\in\ah^* \st \lam\notin \wlat_{\ag}^{\ge}\ \text{and}\ \mdl_\lam\subset \mdc} \in \bounded\big/(\ra)
	\end{equation}
	of highest weights of \infdim{} highest-weight submodules of $\mdc$.
	\item Every connected component $[\lam] \in \bounded \big/ (\ra)$ has $2$ elements and the intersection $(\wgrp \cdot \lam) \cap \bounded$ is always a single connected component in $\bounded$.
	\end{enumerate}
\end{proposition}
\noindent We conclude that an irreducible semisimple coherent family of $\SLA{sp}{2n}$-modules is always completely characterised by its central character.

\section{The combinatorics of classifying weight modules} \label{sec:refine}

To apply our classification in concrete examples, we first need to clarify which (infinite-dimensional) \hw{} $\ag$-modules appear in any given parabolic family.  Recall that the Weyl group $\wgrp$ of $\ag$ is generated by the reflections $\wref{\alpha}$, $\alpha \in \roots{\ag}$.  We recall the definition of the small Weyl group, following \cite{DimStr00}.
\begin{definition}
	Given a simple weight $\ag$-module $\mdm$, let $\nilroots{\mdm}$ denote the set of roots $\alpha \in \roots{\ag}$ whose positive and negative root vectors act locally nilpotently on $\mdm$.  The \emph{small Weyl group} $\swgrp{\mdm}$ of $\mdm$ is then the subgroup of $\wgrp$ generated by the $\wref{\alpha}$ with $\alpha \in \nilroots{\mdm}$.
\end{definition}
\noindent The small Weyl group of a simple dense $\ag$-module is therefore trivial because all root vectors act injectively.  It is easy to see that the action of a root vector on any simple $\ag$-module is either injective or locally nilpotent (the set of vectors on which the action is locally nilpotent is a submodule).

An easy way to appreciate the small Weyl group is to look at the case in which $\mdm$ is a simple \hw{} $\SLA{sl}{2}$-module, with respect to some Borel subalgebra $\ab$.  Then, there are two possibilities:
\begin{enumerate}
	\item $\mdm$ is \fdim{}, so $\nilroots{\mdm} = \roots{\SLA{sl}{2}}$ and $\swgrp{\mdm} = \wgrp \cong \ZZ_2$.
	\item $\mdm$ is \infdim{}, so $\nilroots{\mdm} = \varnothing$ and $\swgrp{\mdm} = \wun$.
\end{enumerate}
There are of course $\abs{\wgrp} = 2$ choices of Borel (containing our fixed Cartan subalgebra $\ah$).  In the first case, $\mdm$ is \hw{} with respect to either choice of Borel; in the second, only one choice makes $\mdm$ \hw{}.  Now consider this from the perspective of the parabolics (in this case, Borels).  A simple \hwm{} $\mdm$ has the form $\pind{\ab} \CC_{\lam}$, for some simple $\ah$-module $\CC_{\lambda}$ ($\lambda$ is the highest weight of $\mdm$), and some Borel $\ab$.  If $\mdm$ is \fdim{}, then it is also \hw{} with respect to the other Borel $\wref{}(\ab)$, though its highest weight is no longer $\lambda$ but $\wref{}(\lambda) = -\lambda$.  Thus, $\pind{\ab} \CC_{\lam} \cong \mdm \cong \pind{\wref{}(\ab)} \CC_{\wref{}(\lam)}$ when $\wref{} \in \swgrp{\mdm}$.

In general, the small Weyl group describes exactly this lack of uniqueness in representing a simple module through parabolic induction.  Recall from \cref{thm:fernando} that every simple weight $\ag$-module $\mdm$, with \fdim{} weight spaces, has the form $\mdm\cong \pind{\ap}\mdn$, for some parabolic subalgebra $\ap \subseteq \ag$ and some simple dense module $\mdn$ over the Levi factor $\al$ of $\ap$.
\begin{lemma}[Dimitrov--Mathieu--Penkov \protect{\cite[Thm.~6.1]{DimStr00}}] \label{lem:dimitrov}
	Given a simple weight $\ag$-module $\mdm$, with \fdim{} weight spaces, the choice of $(\ap,\mdn)$ is unique up to the action of the small Weyl group $\swgrp{\mdm}$.
\end{lemma}
\noindent In other words, if we also have $\mdm\cong \pind{\ap'}\mdn'$ for some parabolic $\ap' \subseteq \ag$ and some simple dense module $\mdn'$ over the Levi factor of $\ap'$, then there exists $\wref{} \in \swgrp{\mdm}$ such that $\ap'=\wref{}(\ap)$ and $\mdn'\cong \wref{}(\mdn)$.

At this point, it is convenient to describe an often more practical means of computing the small Weyl group of a simple weight $\ag$-module $\mdm$.  Let $\ap \subseteq \ag$ be a parabolic subalgebra with Levi factor $\al$.  We choose a set of simple roots $\sroots{\ag}$ of $\ag$ such that the corresponding root vectors all belong to $\ap$.  This ensures, in particular, that $\sroots{\ag}$ includes a set $\sroots{\al}$ of simple roots of $\al$.  Define $\sroots{\al}^{\perp}$ to be the subset of $\sroots{\ag}$ consisting of the simple roots that are orthogonal to those of $\al$.  The simple coroot $\scroot{i}$ corresponding to any $\sroot{i} \in \sroots{\al}^{\perp}$ therefore acts on $\uinv{\ap} \mdm$ as multiplication by $\lambda_i = \bilin{\lambda}{\scroot{i}}$, where $\lambda$ is any weight of $\supp{\uinv{\ap} \mdm}$, by \cref{prop:functors}\ref{it:functlex} and Schur's lemma.
\begin{proposition}[Mathieu \protect{\cite[Prop.~1.3(ii)]{MatCla00}}] \label{prop:swg}
	Suppose that $\mdm \cong \pind{\ap} \mdn$, for some simple dense $\al$-module $\mdn$.  Then, the small Weyl group $\swgrp{\mdm}$ is the subgroup of $\wgrp$ generated by the simple Weyl reflections $\swref{i}$ with $\sroot{i} \in \sroots{\al}^{\perp}$ and $\lambda_i \in \NN$.
\end{proposition}
\begin{proof}
	As we were not able to find a detailed proof of this useful result in the literature, we provide one for the reader's convenience.  Let $\sroots{}(\mdm)$ denote the set of $\sroot{i} \in \sroots{\al}^{\perp}$ for which $\lambda_i \in \NN$.  This set of simple roots corresponds to a Lie subalgebra $\ag(\mdm)$ of $\ag$ with root system $\roots{}(\mdm)$.  We shall prove the \lcnamecref{prop:swg} by showing that $\roots{}(\mdm) = \nilroots{\mdm}$.

	Suppose first that $\alpha$ is a positive root in $\nilroots{\mdm}$.  As the root vectors of $\al$ act injectively on the simple dense $\al$-module $\uinv{\ap}{\mdm} \subset \mdm$, we must have $e^{\alpha} \in \au$.  Thus, $e^{\alpha}$ annihilates any weight vector $v \in \uinv{\ap} \mdm$.  If $\lambda$ is the weight of $v$, then the action of $f^{\alpha}$ on $v$ is only locally nilpotent when $\bilin{\lambda}{\coroot{\alpha}} \in \NN$.

	If there exists $\beta \in \roots{\al}$ with $\bilin{\beta}{\coroot{\alpha}} \neq 0$, then without loss of generality we may assume that this quantity is positive.  Since $f^{\beta}$ acts injectively, it follows that $(f^{\beta})^n v$ is a non-zero element of the $\al$-module $\uinv{\ap}{\mdm}$, for any $n \in \NN$.  The actions of $e^{\alpha}$ and $f^{\alpha}$ on $(f^{\beta})^n v = 0$ are therefore zero and locally nilpotent, respectively, for all $n \in \NN$, hence we must have $\bilin{\lambda - n \beta}{\coroot{\alpha}} \in \NN$ for all $n \in \NN$.  Since $\bilin{\beta}{\coroot{\alpha}} > 0$, this is a contradiction, proving that $\bilin{\beta}{\coroot{\alpha}} = 0$ for all $\beta \in \roots{\al}$.

	It remains to show that $\alpha$ is a linear combination of simple roots in $\roots{}(\mdm)$.  For this, we induct over the height of $\alpha$.  If the height is $1$, then $\alpha$ is simple and the conditions established above prove that it is in $\roots{}(\mdm)$.  We may therefore assume that the height of $\alpha$ is greater than $1$, so $\alpha = \beta + \gamma$ for some $\beta, \gamma \in \roots{\ag}$, and that the statement has been proven for all roots in $\nilroots{\mdm}$ of lower height than $\alpha$.  There are three cases to consider:
	\begin{enumerate}[label=\arabic*)]
		\item $f^{\beta}$ and $f^{\gamma}$ act injectively on $\mdm$.  Then, $-\alpha = -\beta - \gamma \in C_{\mdm}$, the cone (monoid) generated by the roots whose root vectors act injectively on $\mdm$.  But, this contradicts \cite[Lem.~5.1(iii)]{DimStr00} which says that $C_{\mdm}$ has trivial intersection with the root lattice $\nilrlat{\mdm}$ generated by $\nilroots{\mdm}$.  This case is therefore impossible.
		\item $f^{\beta}$ acts injectively on $\mdm$ whilst $f^{\gamma}$ acts locally nilpotently.  But then, $-\beta = \gamma - \alpha \in \nilrlat{\mdm}$ which again contradicts \cite[Lem.~5.1(iii)]{DimStr00} and is thus impossible.
		\item $f^{\beta}$ and $f^{\gamma}$ act locally nilpotently on $\mdm$.  Then, $\beta, \gamma \notin \roots{\al}$ so $e^{\beta}, e^{\gamma} \in \au$ and so $\beta, \gamma \in \nilroots{\mdm}$.  By induction, $\beta$ and $\gamma$ are linear combinations of simple roots in $\roots{}(\mdm)$ and hence so is $\alpha$.
	\end{enumerate}
	This establishes that $\nilroots{\mdm} \subseteq \roots{}(\mdm)$.

	For the opposite inclusion, note that it is enough to prove that $e^{\alpha}$ and $f^{\alpha}$, for $\alpha \in \roots{}(\mdm)$ positive, act locally nilpotently on a single non-zero vector $v \in \mdm$, since $\mdm$ is simple (\cref{prop:functors}\ref{it:functrex}).  Noting that $\alpha$ being orthogonal to $\roots{}(\mdm)$ implies that $e^{\alpha} \in \au$, we may take $v \in \uinv{\ap}{\mdm}$ so that $e^{\alpha} v = 0$.

	To show that $f^{\alpha}$ acts locally nilpotently on $v$ requires more work.  We will actually show something a little stronger, namely that the $\ag(\mdm)$-module $\mdd$ generated from $v$ is \fdim.  Let $\lambda$ be the weight of $v$ and recall that $\lambda_i \in \NN$ for all $\sroot{i} \in \sroots{}(\mdm)$.  We will prove that the only $w \in \mdd$ which are annihilated by the $e^{\beta}$, with $\beta \in \roots{}(\mdm)$ positive, are multiples of $v$.

	Let us therefore assume that such a $w$ exists, but is not a multiple of $v$.  It is therefore not in $\uinv{\ap}{\mdm}$.  However, as $\mdm$ is simple (as a $\ag$-module), there exists $U \in \envalg{\ag}$ such that $Uw=v$.  We can write $U$ as a linear combination of monomials which are ordered as follows: root vectors of $\au^-$ appear to the left of negative root vectors of $\al$, which appear to the left of positive root vectors of $\al$; finally, root vectors of $\au$ appear to the right.  As $v \in \uinv{\ap}{\mdm}$, we may in fact assume that no root vectors of $\au^-$ appear.  Similarly, as $w \notin \uinv{\ap}{\mdm}$, we may assume that at least one root vector of $\au$ appears in each monomial.  This may be sharpened by partitioning the positive roots of $\ag$ into those of $\al$, those of $\ag(\mdm)$ and the remainder $\tilde{\roots{}}{\mspace{-2mu}}^+$.  As the positive root vectors of $\ag(\mdm)$ annihilate $w$ by hypothesis, we may thus assume that at least one $e^{\beta}$, with $\beta \in \tilde{\roots{}}{\mspace{-2mu}}^+$, appears in each monomial.

	It follows from the orthogonality of $\sroots{\al}$ and $\sroots{}(\mdm)$ that any given $\beta \in \tilde{\roots{}}{\mspace{-2mu}}^+$ may be decomposed as $\sroot{i} + \beta'$, where $\sroot{i}$ is in the complement $\tilde{\sroots{}}$ of $\sroots{\al} \sqcup \sroots{}(\mdm)$ in $\sroots{\ag}$ and $\beta'$ is a non-negative-integer linear combination of simple roots.  The weight of $U$ is therefore a linear combination of simple roots in which the coefficient of $\sroot{i}$ is positive.  However, this contradicts $w \in \mdd = \envalg{\ag(\mdm)} v$ because the latter implies that only the coefficients of the simple roots of $\ag(\mdm)$ can be non-zero.  This contradiction establishes that the vector $w$ does not exist, hence that $f^{\alpha}$ acts locally nilpotently on $v$ as required.
\end{proof}

We note that computing $\swgrp{\mdm}$ can be done directly at the level of the Dynkin diagrams $\dynkin{\ag}$ and $\dynkin{\as}$ of $\ag$ and the semisimple subalgebra $\as = \comm{\al}{\al}$ of $\al$, respectively.  The latter is of course the subdiagram of $\dynkin{\ag}$ consisting of the nodes corresponding to the simple roots $\sroots{\al} \subseteq \sroots{\ag}$ and the edges connecting them.  The simple roots $\sroots{\al}^{\perp}$ thus correspond to the nodes of $\dynkin{\ag}$ that are neither in $\dynkin{\as}$ nor are directly connected to any node in $\dynkin{\as}$.  Moreover, for each such node, the scalar $\lambda_i$ is just the (necessarily common) Dynkin label of the weights of $\uinv{\ap} \mdm$.

We illustrate this with a simple example: $\ag = \SLA{sl}{4}$ and $\mdm = \pind{\ap} \mdn$, where $\ap$ is the parabolic subalgebra of $\ag$ corresponding to $\sroot{1}$ and $\mdn$ is a simple dense module over the Levi factor $\al \cong \SLA{sl}{2} \oplus \SLA{gl}{1}^{\oplus 2}$.  Since the Dynkin diagram of $\as \cong \SLA{sl}{2}$ is realised as the first node of that of $\SLA{sl}{4}$, we see that $\sroots{\al}^{\perp} = \set{\sroot{3}}$ as the second node is directly connected to the first.  If the third Dynkin label of any (and thus every) weight of $\mdn$ is a non-negative integer, then the small Weyl group of $\mdm$ is $\swgrp{\mdm} = \ang{\swref{3}} \cong \ZZ_2$; otherwise, it is $\wun$.

Given a semisimple weight module $\mdm=\bigoplus_i \mdm_i$, where the $\mdm_i$ are simple and weight with \fdim{} weight spaces, we define the small Weyl group of $\mdm$ to be $\swgrp{\mdm} = \bigcap_i \swgrp{\mdm_i}$.  In particular, consider the small Weyl group of an irreducible semisimple parabolic family $\mdp = \pind{\ap}{\mdc}$ of $\ag$-modules, where $\ap$ has non-abelian Levi factor $\al$ and $\mdc$ is a coherent family of $\al$-modules.  The following \lcnamecref{thm:type} now follows from \cref{thm:fernando,lem:dimitrov}.
\begin{proposition}\label{thm:type}
	Given a parabolic family $\mdp$ of $\ag$-modules, the choice of $(\ap,\mdc)$ is unique up to the action of the small Weyl group $\swgrp{\mdp}$.
\end{proposition}

As one might hope, the simple dense submodules of a parabolic family of $\ag$-modules all have the same small Weyl group.  We may therefore compute $\swgrp{\mdp}$ using the method discussed above.  More importantly, we do not have to perform uncountably many computations in order to deduce the small Weyl groups of all its simple submodules.
\begin{proposition}\label{prop:smallweyl}
	Let $\mdn$ be a simple dense $\al$-submodule of $\mdc$ and $\mdm=\pind{\ap}\mdn$ the corresponding simple submodule of $\mdp = \pind{\ap} \mdc$.  Then, the small Weyl group of $\mdm$ is contained in the small Weyl group of every submodule of $\mdp$.  In particular, $\swgrp{\mdp}=\swgrp{\mdm}$.
\end{proposition}
\begin{proof}
	Choose a positive root $\alpha \in \nilroots{\mdm}$.  Then, $e^{\alpha} \in \au$ so $e^{\alpha} v = 0$ for all $v \in \uinv{\ap} \mdm = \mdn$.  Moreover, because the action of $f^{\alpha}$ on $v$ is locally nilpotent, $\coroot{\alpha} = \comm{e^{\alpha}}{f^{\alpha}}$ acts on $v$ as multiplication by some non-negative integer $\lambda_{\alpha}$.  This integer is $v$-independent because $\coroot{\alpha}$ is orthogonal to $\roots{\al}$ (\cref{prop:swg}).  As $\mdn$ is dense, it is bounded and so its weights are Zariski-dense in $\supp{\mdc}$ (\cref{thm:mathieu1}\ref{it:zariski}).  Because $\coroot{\alpha} \in \ah \subset \cent{\ah}{\al}$, it acts polynomially on $\mdc$.  We therefore conclude that $\coroot{\alpha}$ acts as multiplication by $\lambda_{\alpha}$ on all of $\mdc$.

	In particular, $\coroot{\alpha}$ acts as multiplication by $\lambda_{\alpha} \in \NN$ on any simple submodule $\mdn' \subset \mdc = \uinv{\ap} \mdp$.  Thus, $e^{\alpha}$ acts on $\mdn'$ as $0$ and so $f^{\alpha}$ acts locally nilpotently on the subspace $\mdn'$ of the $\ag$-module $\mdm' = \pind{\ap} \mdn' \subset \mdp$.  As the action of a root vector on a simple module is either injective or locally nilpotent, this proves that $e^{\alpha}$ and $f^{\alpha}$ both act locally nilpotently on $\mdm$'.  In other words, $\alpha \in \nilroots{\mdm'}$ and so $\nilroots{\mdm} \subseteq \nilroots{\mdm'}$.  The small Weyl group of every simple submodule of $\mdp$ thus contains that of $\mdm$, completing the proof.
\end{proof}

\section{A classification algorithm} \label{sec:algorithm}

We shall now combine \cref{thm:classification} with the theory developed in \cref{sec:ac,sec:refine} to present an algorithm whose input is the classification of simple \hw{} $\aaz$-modules and whose output is the classification of all simple weight $\aaz$-modules with \fdim{} weight spaces.  In \cref{sec:ex} below, we shall illustrate this algorithm with several examples in which $\aaz$ is the Zhu algebra $\aaz_{\kk}$ of a simple affine \voa{} $\lkg$.  In this case, the algorithm then implies the classification of the simple relaxed \hw{} $\lkg$-modules, by \cref{thm:zhu}\ref{it:zhusimp}, again assuming \fdim{} weight spaces.

Fix a set of simple roots $\sroots{\ag} = \set{\sroot{1},\dots,\sroot{r}}$ of $\ag$, where $r$ is the rank of $\ag$.  We shall refer to a parabolic subalgebra of $\ag$ as being \emph{standard} (with respect to $\sroots{\ag}$) if it contains all of the simple root vectors $e^{\sroot{i}}$, $i=1,\dots,r$.  We shall similarly call a parabolic family $\mdp = \pind{\ap} \mdc$ \emph{standard} when it is induced from a coherent family $\mdc$ over the Levi factor $\al$ of a standard parabolic $\ap$.  We recall that a standard parabolic subalgebra $\ap \subseteq \ag$ is completely determined by the set $S \subseteq \set{1,\dots,r}$ of indices $i$ (or Dynkin nodes) for which the negative simple root vectors $f^{\sroot{i}}$ also belong to $\ap$.

Just as every parabolic subalgebra of $\ag$ may be obtained from a standard parabolic subalgebra by acting with the Weyl group $\wgrp$, every parabolic family may similarly be obtained from a standard parabolic family using $\wgrp$.  Since coherent families of $\al$-modules are invariant under the action of the Weyl group $\wgrp_{\al} \subseteq \wgrp$ of $\al$, as is $\ap$, it follows that the parabolic family $\mdp$ is also preserved by $\wgrp_{\al}$.  Moreover, the small Weyl group $\swgrp{\mdp}$ preserves $\mdp$ but not necessarily $\mdc$ or $\ap$.  Thus, the $\wgrp$-orbit of each standard parabolic family $\mdp$ gives $\abs{\wgrp} \big/ \brac*{\abs{\swgrp{\mdp}} \abs{\wgrp_{\al}}}$ different parabolic families.  Indeed, the classification of parabolic families of $\ag$-modules reduces to that of standard parabolic families and the computation of their small Weyl groups (see \cref{thm:type,prop:smallweyl}).

The basic idea of the classification algorithm is to choose a standard parabolic subalgebra $\ap \subseteq \ag$ and determine which, if any, of the simple \hw{} $\aaz$-modules are $\al$-bounded.  By \cref{thm:classification}, each such module is contained in an irreducible semisimple standard parabolic family of $\aaz$-modules and every simple weight $\aaz$-module, with \fdim{} weight spaces, is contained in a $\wgrp$-twist of such a parabolic family.  To assist with determining when a \hwm{} is $\al$-bounded, write
\begin{equation}
	\al = \as_1 \oplus \dots \oplus \as_m \oplus \az,
\end{equation}
where the $\as_i$ are simple ideals and $\az$ is the centre of $\al$.  We let $\pi_{\as_i}$ denote the orthogonal projection onto $\as_i$ and let $\CC_{\mu}$ denote the one-dimensional $\az$-module whose sole weight is $\mu$.

The classification algorithm is then as follows.
\begin{algorithm}
	Let $\ag$ be a \fdim{} simple Lie algebra and let $\aaz$ be a quotient of $\envalg{\ag}$ by a two-sided ideal.  Assume that the simple \hw{} $\aaz$-modules have been classified.
	\begin{itemize}
		\item Consider each non-empty subset $S \subseteq \set{1,\dots,r}$ and determine if the corresponding standard parabolic subalgebra $\ap$ is of \type{AC}-type.  This is easy to check by looking at the connected components of the Dynkin diagram of $\as$, where $\as = \comm{\al}{\al}$ and $\al$ is the Levi factor of $\ap$.
		\item If $\ap$ is of \type{AC}-type, consider the highest weight $\lambda$ of each simple \hw{} $\aaz$-module $\mdh$ and compute the projections $\pi_{\as_i}(\lambda)$, $i=1,\dots,m$, onto the weight spaces of the simple ideals $\as_i$ of $\al$.
		\item For each $i=1,\dots,m$, use \cref{prop:bounded-a,prop:bounded-c} to determine whether $\pi_{\as_i}(\lambda) \in \bounded_{\as_i}$.  If so, then there is an irreducible semisimple coherent family $\mdc_i$ of $\as_i$-modules containing the simple \hw{} $\as_i$-module of highest weight $\pi_{\as_i}(\lambda)$.
		\item If $\mdc_i$ exists for all $i=1,\dots,m$, then there is an irreducible semisimple standard parabolic family
		\begin{equation}
			\mdp = \pind{\ap} \brac*{\mdc_1 \otimes \dots \otimes \mdc_m \otimes \CC_{\mu}}, \qquad
			\mu = \lambda - \sum_{i=1}^m \pi_{\as_i}(\lambda),
		\end{equation}
		that contains $\mdh$.  $\mdp$ is thus a $\aaz$-module, by \cref{thm:classification}, hence so are all its direct summands.
		\item Determine which $\lambda$ give the same parabolic family by using \cref{prop:bij-a,prop:connect-a,prop:connect-c} to compute the connected components $[\pi_{\as_i}(\lambda)] \in \bounded_{\as_i} \big/ (\rightarrow)$.
		\item For each irreducible semisimple standard parabolic family $\mdp$ of $\aaz$-modules found, act with representatives of $\wgrp \big/ \brac*{\swgrp{\mdp} \times \wgrp_{\al}}$ to obtain a complete set of irreducible semisimple parabolic families of $\aaz$-modules.
	\end{itemize}
	Along with the simple \hw{} $\aaz$-modules, the direct summands of the irreducible semisimple parabolic families of $\aaz$-modules found with this algorithm form a complete set, up to isomorphism, of simple weight $\aaz$-modules with \fdim{} weight spaces.
\end{algorithm}

\section{Examples} \label{sec:ex}

In this \lcnamecref{sec:ex}, we apply the classification algorithm to some concrete examples of simple \voas{} $\lkg$ in order to classify the irreducible semisimple standard parabolic (and coherent) families of the corresponding Zhu algebras $\aaz_{\kk} = \zhu{\lkg}$.  By \cref{thm:zhu}\ref{it:zhusimp}, this yields a classification of all the simple relaxed \hwms{} (with \fdim{} weight spaces) of the \voa{}.  We recall that non-standard parabolic families are obtained by twisting standard ones by elements of the Weyl group of $\ag$, as described in \cref{thm:type,prop:smallweyl}.  We use the same notations as in \cref{sec:refine} and will assume that all the parabolic families considered in this \lcnamecref{sec:ex} are both semisimple and irreducible.

\subsection{Example: $\slvoa{\kk}{2}$ for $\kk$ admissible} \label{ex:a1}

We warm up with the familiar case of $\ag = \SLA{sl}{2}$ with $\kk$ admissible and non-integral:
\begin{equation}
	k+2 = \frac{u}{v}, \qquad \text{where}\ u,v \in \ZZ_{\ge 2}\ \text{and}\ \gcd \set{u,v} = 1.
\end{equation}
Let $\sroot{1}$ denote the simple root of $\SLA{sl}{2}$, so that $\fwt{1} = \frac{1}{2} \sroot{1}$ is the fundamental weight.

The simple \hw{} $\slvoa{\kk}{2}$-modules were originally classified in \cite{AdaVer95,DonVer97}, see also \cite{RidRel15}.  Their Zhu images are the \hw{} $\SLA{sl}{2}$-modules $\mdl_{r,s}$ of highest weights $\lambda_{r,s} = (r-1-\frac{u}{v} s) \fwt{1}$, where $r=1,2,\dots,u-1$ and $s=0,1,\dots,v-1$.  Note that the $\mdl_{r,s}$ with $s>0$ are bounded, so $\lambda_{r,s} \in \bounded_{\SLA{sl}{2}}$ for all $s>0$.  As these $\lambda_{r,s}$ are never integral, the connected component $[\lambda_{r,s}]$ has two elements (\cref{prop:connect-a}): $\lambda_{r,s}$ and $\swref{1} \cdot \lambda_{r,s} = \lambda_{u-r,v-s}$.  (Here, $\swref{1}$ denotes the simple Weyl reflection of $\SLA{sl}{2}$.)

Each distinct connected component $[\lambda_{r,s}]$, for $r=1,2,\dots,u-1$ and $s=1,2,\dots,v-1$, therefore gives rise to a distinct (standard) coherent family $\mdc_{[r,s]}$ of $\aaz_{\kk}$-modules, making $\frac{1}{2} (u-1)(v-1)$ families in all.  As coherent families are invariant under the action of the Weyl group $\wgrp \cong \ZZ_2$, there is no need to consider non-standard families.  Moreover, the $\mdc_{[r,s]}$ are distinguished by their central characters (\cref{prop:connect-a} again), meaning the eigenvalues of the quadratic Casimir.  Along with the $\mdl_{r,0}$, $r=1,2,\dots,u-1$, the simple direct summands of the $\mdc_{[r,s]}$ exhaust the simple weight $\aaz_{\kk}$-modules with \fdim{} weight spaces.  Note that because coherent families are $\wgrp$-invariant, the \lwms{} $\swref{1}(\mdl_{r,s})$ and $\swref{1}(\mdl_{u-r,v-s})$, with $s>0$, are also contained in $\mdc_{[r,s]}$ and are thus also $\aaz_{\kk}$-modules.

The corresponding $\slvoa{\kk}{2}$-modules therefore provide a classification of the simple relaxed \hwms{} (with \fdim{} weight spaces).  Each of these coherent families of $\slvoa{\kk}{2}$-modules is determined by its conformal weight $\Delta_{r,s}$, which is proportional to the eigenvalue of the quadratic Casimir of $\SLA{sl}{2}$ on the Zhu image $\mdc_{[r,s]}$:
\begin{equation}
	\Delta_{r,s} = \frac{1}{2(\kk+2)} \bilin{\lambda_{r,s}}{\lambda_{r,s} + 2 \wvec} = \frac{v}{2u} \frac{(r-1-\frac{u}{v}s) (r+1-\frac{u}{v}s)}{2} = \frac{(vr-us)^2-v^2}{4uv}.
\end{equation}
These results reproduce exactly the known classification of relaxed \hw{} $\slvoa{\kk}{2}$-modules that was obtained in \cite{AdaVer95,RidRel15} using more arduous methods.

\subsection{Example: $\slvoa{-3/2}{3}$} \label{ex:a2}

We next consider $\ag=\SLA{sl}{3}$ with simple roots $\sroot{1}$ and $\sroot{2}$, giving fundamental weights $\fwt{1}=\frac{1}{3}(2\sroot{1}+\sroot{2})$ and $\fwt{2}=\frac{1}{3}(\sroot{1}+2\sroot{2})$.  The level in this example is $\kk = -\frac{3}{2}$ which is admissible.  Moreover, $\slvoa{-3/2}{3}$ is the second member of a family of \voas{} related to the Deligne exceptional series --- the first being $\slvoa{-4/3}{2}$ --- that have recently attracted much attention in mathematics and physics, see \cite{BeeInf15,AraQua16} for example.

The simple \hw{} $\slvoa{-3/2}{3}$-modules were originally classified in \cite{PerVer08}.  The corresponding simple modules over the Zhu algebra $\aaz_{-3/2}$ are as follows:  One \fdim{} \hwm{} $\mdl_0$ and three \infdim{} \hwms{} $\mdl_{\Lam_1}$, $\mdl_{\Lam_2}$, $\mdl_{\Lam_3}$, where $\Lam_1 = -\frac{3}{2} \fwt{1}$, $\Lam_2 = -\frac{3}{2} \fwt{2}$ and $\Lam_3 = -\frac{1}{2} (\fwt{1}+\fwt{2})$.  Here, the subscripts indicate the highest weight.

We will now extend this to a classification of standard parabolic families of $\aaz_{-3/2}$-modules.  Recall that a standard parabolic subalgebra is determined by the subset $S$ of $\set{1,2}$ corresponding to which negative simple root vectors it contains.  When $S$ is empty, the parabolic is the standard Borel and so the corresponding parabolic families are just the \hw{} $\aaz_{-3/2}$-modules given above.

At the other extreme, $S=\set{1,2}$ corresponds to $\ap=\al=\ag$ and so parabolic families reduce to coherent families.  It is easy to check from \cref{prop:bounded-a} that the highest weights $\Lam_i$, $i=1,2,3$, of the \infdim{} \hw{} $\slvoa{-3/2}{3}$-modules listed above all belong to $\bounded_{\SLA{sl}{3}}$.  Indeed, $\Lam_1$ and $\Lam_2$ satisfy condition \ref{it:Aedge} while $\Lam_3$ satisfies condition \ref{it:Adoub}.  None of these weights are integral, so each belongs to a connected component of $\bounded_{\SLA{sl}{3}} \big/ (\ra)$ with three elements (\cref{prop:connect-a}).  There is therefore only one connected component, hence only one coherent family $\mdc$ of $\aaz_{-3/2}$-modules.  It is characterised by its central character which coincides with the common central character of the $\Lam_i$.  Moreover, $\mdl_{\Lam_i} \subset \mdc$, for each $i=1,2,3$.

Next, set $S=\set{1}$, which corresponds to $\al \cong \SLA{sl}{2} \oplus \SLA{gl}{1}$.  We orthogonally project each of the $\Lam_i$ onto the weight space of $\as \cong \SLA{sl}{2}$, here realised as $\CC \sroot{1}$, obtaining
\begin{equation}
	\Lam_1 = -\frac{3}{4} \sroot{1} - \frac{3}{4} \fwt{2}, \quad \text{hence} \quad \pi_{\as}(\Lam_1) = -\frac{3}{2} \fwt{1}^{\as},
\end{equation}
and similarly $\pi_{\as}(\Lam_2)=0$ and $\pi_{\as}(\Lam_3)=-\frac{1}{2}\fwt{1}^{\as}$ (see \cref{fig:sl3proj}).  Here, $\pi_{\as}$ denotes the orthogonal projection and $\fwt{1}^{\as}$ denotes the fundamental weight of $\as$.  We find that $\pi_{\as}(\Lam_1), \pi_{\as}(\Lam_3)\in\bounded_{\SLA{sl}{2}}$, while $\pi_{\as}(\Lam_2)\notin\bounded_{\SLA{sl}{2}}$.  In other words, the $\slvoa{-3/2}{3}$-modules $\mdl_{\Lam_1}$ and $\mdl_{\Lam_3}$ are $\al$-bounded, hence they correspond to parabolic families.  In fact, they correspond to the same parabolic family because $\pi_{\as}(\Lam_1)$ and $\pi_{\as}(\Lam_3)$ belong to the same connected component in $\bounded_{\SLA{sl}{2}} \big/ (\ra)$.

\begin{figure}
	\begin{tikzpicture}[>=stealth,scale=1.9]
		\draw[red,<->] (90:0.5) -- node[pos=0.03,above left] {\textcolor{black}{$\scriptstyle \lam_2$}} (-90:2);
		\draw[red,<->] (30:0.5) -- node[pos=0.01,right] {\textcolor{black}{$\scriptstyle \lam_1$}} (-150:2);
		\node[bwt,label=below:{$\Lam_1$}] (a1) at (30:-1.5) {};
		\node[bwt,label=below:{$\Lam_3$}] (a3) at ($(30:-0.5) + (90:-0.5)$) {};
		\node[bwt,label=right:{$\Lam_2$}] (a2) at (90:-1.5) {};
		\draw[thick,<->] (180:2) -- node[pos=0.99,below right] {$\CC \sroot{1}$} (0:0.75);
		\node[bwt,label=above:{$\pi_{\as}(\Lam_1)$}] (b1) at ($(30:-0.75) + (-30:-0.75)$) {};
		\node[bwt,label=above:{$\pi_{\as}(\Lam_3)$}] (b3) at ($(30:-0.25) + (-30:-0.25)$) {};
		\node[bwt,label=below right:{$\pi_{\as}(\Lam_2)$}] (b2) at (0,0) {};
		\draw[very thick,dotted,->] (a1) -- (b1);
		\draw[very thick,dotted,->] (a3) -- (b3);
		\draw[very thick,dotted,->] (a2) -- (b2);
	\end{tikzpicture}
\caption{Projecting the $\SLA{sl}{3}$-weights $\Lam_i$, $i=1,2,3$, onto the weight space $\CC \sroot{1}$ corresponding to $S = \set{1}$.  $\lambda_1$ and $\lambda_2$ indicate the directions of increasing Dynkin labels.} \label{fig:sl3proj}
\end{figure}
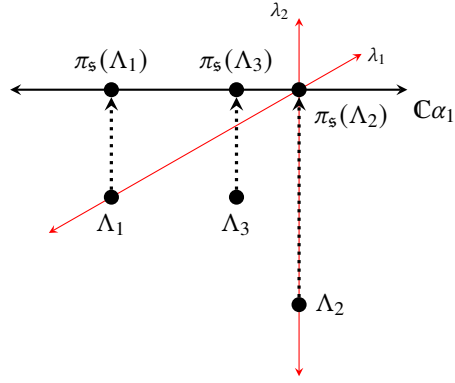

Thus, there is just one standard parabolic family $\mdp^1$ of $\aaz_{-3/2}$-modules corresponding to $S=\set{1}$ and it contains both $\mdl_{\Lam_1}$ and $\mdl_{\Lam_3}$.  It is induced from the coherent family $\mdc_q\otimes \CC_\mu$ of $\al$-modules, where
\begin{equation}
	\begin{aligned}
		q &= \bilin{\pi_{\as}(\Lam_1)}{\pi_{\as}(\Lam_1) + 2 \wvec^{\as}} = \bilin{\pi_{\as}(\Lam_3)}{\pi_{\as}(\Lam_3) + 2 \wvec^{\as}} = -\frac{3}{8} \\
		\text{and} \qquad \mu &= \Lam_1 - \pi_{\as}(\Lam_1) = \Lam_3 - \pi_{\as}(\Lam_3) = -\frac{3}{4} \fwt{2}
	\end{aligned}
\end{equation}
are the eigenvalue of the quadratic Casimir (central character) of $\SLA{sl}{2}$ and the $\SLA{gl}{1}$-weight, respectively.

Similarly, $S=\set{2}$ also yields precisely one standard parabolic family $\mdp^2=\pind{\ap}(\mdc_{-3/8}\otimes \CC_{-3\fwt{1}/4})$. It contains both $\mdl_{\Lam_2}$ and $\mdl_{\Lam_3}$, hence $\mdp^2 \subset \mdc$ as well.  Clearly, this parabolic family may be obtained from that found when $S=\set{1}$ by twisting by the conjugation automorphism (the outer automorphism of $\SLA{sl}{3}$ that acts as $-1$ on $\ah$).

Thus, for a given Borel subalgebra, there is $1$ \fdim{} $\aaz_{-3/2}$-module $\mdl_0$; $3$ \infdim{} \hw{} $\aaz_{-3/2}$-modules $\mdl_{\Lam_i}$, $i=1,2,3$; $2$ standard parabolic families $\mdp^1$ and $\mdp^2$ of $\aaz_{-3/2}$-modules corresponding to $\al \cong \SLA{gl}{2}$; and $1$ coherent family $\mdc$ of $\aaz_{-3/2}$-modules.  The small Weyl groups are as follows.
\begin{center}
	\begin{tabular}{C|CCCCCC}
		\mdm & \mdl_0 & \mdl_{\Lam_1} & \mdl_{\Lam_2} & \mdl_{\Lam_3} & \mdp^1,\ \mdp^2 & \mdc \\
		\hline
		\swgrp{\mdm} & \wgrp & \ang*{\swref{2}} & \ang*{\swref{1}} & \wun & \wun & \wun
	\end{tabular}
\end{center}
Twisting $\mdl_0$ by $\wgrp$, which amounts to changing the Borel, thus leads to $\abs{\wgrp / \swgrp{\mdl_0}} = 1$ \fdim{} simple \hwm{}.  Similarly, we get $3$ twists each for $\mdl_{\Lam_1}$ and $\mdl_{\Lam_2}$, while $\mdl_{\Lam_3}$ gets $6$.  The action of $\wgrp$ on the parabolic corresponding to $S=\set{1}$ results in $\abs{\wgrp / \wgrp_{\SLA{sl}{2}}} = 3$ parabolics, because the parabolic families are invariant under acting with the Weyl group of $\al$ (which coincides with that of $\SLA{sl}{2}$).  There are another $3$ coming from $S=\set{2}$, hence we have $6$ parabolic families in total.  Finally, there is only a single coherent family.

The resulting classification of simple relaxed \hw{} $\slvoa{-3/2}{3}$-modules (with \fdim{} weight spaces) appears to be consistent with the Gelfand-Tsetlin classification reported in \cite{AraWei16}.  Some relaxed \hwms{} for this \voa{} were also constructed in \cite{AdaRea16}, but with no claim of completeness.  An analysis of the characters, modular properties and Grothendieck fusion rules of all these modules will appear in \cite{KawAdm19}.

\subsection{Example: $\spvoa{-1/2}{4}$} \label{ex:c2}

Consider now a non-simply-laced admissible-level example: $\ag=\SLA{sp}{4}$ and $\kk=-\frac{1}{2}$.  Recall from \cref{sec:typec} that we take $\sroot{1}$ to be short and $\sroot{2}$ long, so that the fundamental weights are $\fwt{1}=\sroot{1}+\frac{1}{2}\sroot{2}$ and $\fwt{2}=\sroot{1}+\sroot{2}$.

The simple \hw{} $\spvoa{-1/2}{4}$-modules were first classified in \cite{AdaSom94}.  There turn out to be four simple \hw{} $\aaz_{-1/2}$-modules, of which two are \fdim{} ($\mdl_0$ and $\mdl_{\fwt{1}}$) and two are not.  The highest weights of the \infdim{} modules will be denoted by $\Lam_1=-\frac{1}{2}\fwt{2}$ and $\Lam_2=\fwt{1}-\frac{3}{2}\fwt{2}$.

As always, we work down the list of (standard) parabolic subalgebras (ignoring the Borel case that corresponds to \hwms{}).  Starting with $S=\set{1,2}$, hence $\ap=\al=\ag$, we check that both $\Lam_1$ and $\Lam_2$ satisfy all the conditions of \cref{prop:bounded-c}, hence they belong to $\bounded_{\SLA{sp}{4}}$.  Since connected components for $\SLA{sp}{4}$ always have two elements, by \cref{prop:connect-c}, there is a single connected component giving exactly one coherent family $\mdc$ of $\aaz_{-1/2}$-modules.  It is characterised by its central character and contains both $\mdl_{\Lam_1}$ and $\mdl_{\Lam_2}$.

If $S=\set{1}$, hence $\al \cong \SLA{sl}{2} \oplus \SLA{gl}{1}$, then projecting onto the $\as=\SLA{sl}{2}$ weight space spanned by $\sroot{1}$ results in $\pi_{\as}(\Lam_1)=0$ and $\pi_{\as}(\Lam_2)=\fwt{1}^{\as}$.  As both are dominant integral $\SLA{sl}{2}$-weights, they are not in $\bounded_{\SLA{sl}{2}}$ and there are therefore no parabolic families corresponding to this $S$.

When $S=\set{2}$ however, we again have $\al \cong \SLA{sl}{2} \oplus \SLA{gl}{1}$, but the projection this time gives two elements of $\bounded_{\SLA{sl}{2}}$: $\pi_{\as}(\Lam_1) = -\frac{1}{2} \fwt{1}^{\as}$ and $\pi_{\as}(\Lam_2) = -\frac{3}{2} \fwt{1}^{\as}$.  These represent a single connected component for $\SLA{sl}{2}$, hence it corresponds to a single coherent family $\mdc_{-3/8}\otimes \CC_{-\fwt{1}/2}$ of $\al$-modules.  There is thus a unique standard parabolic family $\mdp=\pind{\ap}(\mdc_{-3/8} \otimes \CC_{-\fwt{1}/2})$ of $\aaz_{-1/2}$-modules.

As always, the small Weyl group of the \fdim{} modules is $\wgrp$.  There are therefore only $2$ \fdim{} simple $\aaz_{-1/2}$-modules.  The small Weyl group of both $\mdl_{\Lam_1}$ and $\mdl_{\Lam_2}$ is $\ang{\swref{1}}$, hence $\wgrp$-twisting gives $\abs{\wgrp / \ang{\swref{1}}} = 4$ modules each.  The result is thus $8$ \infdim{} simple \hw{} $\aaz_{-1/2}$-modules.  Once again, the small Weyl groups of the parabolic and coherent families are trivial.  Hence, we get $\abs{\wgrp / \wgrp_{\SLA{sl}{2}}} = 4$ parabolic families and a single coherent family of $\aaz_{-1/2}$-modules.  We believe that the corresponding classification of simple relaxed \hw{} $\spvoa{-1/2}{4}$-modules, with \fdim{} weight spaces, is new.

\subsection{Example: $\affvoa{-5/3}{\ag_2}$} \label{ex:g2}

Our next example features a simple Lie algebra which is not of \type{AC}-type.  Take $\ag=\ag_2$ and $k=-\frac{5}{3}$, an admissible level corresponding to the third member of the Deligne exceptional series of affine \voas.  Our convention is that $\sroot{1}$ denotes the short simple root and $\sroot{2}$ the long one.  The fundamental weights are thus $\fwt{1}=2\sroot{1}+\sroot{2}$ and $\fwt{2}=3\sroot{1}+2\sroot{2}$.

The classification of simple \hw{} $\affvoa{-5/3}{\ag_2}$-modules was carried out in \cite{AxtVer11}.  At the level of the Zhu algebra $\aaz_{-5/3}$, there is a single simple \fdim{} module $\mdl_0$ and two \infdim{} simple \hwms{} $\mdl_{\Lam_1}$ and $\mdl_{\Lam_2}$, where $\Lam_1=-\frac{2}{3}\fwt{2}$ and $\Lam_2=\fwt{1}-\frac{4}{3}\fwt{2}$.

Because $\ag_2$ is not of \type{AC}-type, there can be no coherent families corresponding to $S=\set{1,2}$.  Moreover, $S=\set{1}$ (thus $\al \cong \SLA{sl}{2} \oplus \SLA{gl}{1}$) does not yield any parabolic families because neither $\pi_{\as}(\Lam_1)=0$ nor $\pi_{\as}(\Lam_2)=\fwt{1}^{\as}$ belong to $\bounded_{\SLA{sl}{2}}$.  We therefore turn to $S=\set{2}$, for which $\al \cong \SLA{sl}{2} \oplus \SLA{gl}{1}$, $\pi_{\as}(\Lam_1)=-\frac{2}{3}\fwt{2}^{\as}$ and $\pi_{\as}(\Lam_2)=-\frac{4}{3}\fwt{2}^{\as}$.  Both projections belong to $\bounded_{\SLA{sl}{2}}$ and constitute a single connected component.  We therefore have only one standard parabolic family $\mdp$ of $\aaz_{-5/3}$-modules (and it corresponds to $S=\set{2}$).  It contains both $\mdl_{\Lam_1}$ and $\mdl_{\Lam_2}$ and is constructed by applying $\pind{\ap}$ to the coherent family $\mdc_{-4/9}\otimes \CC_{-\fwt{1}}$ of $\SLA{sl}{2} \oplus \SLA{gl}{1}$-modules.

Summarising, there is just one \fdim{} simple \hw{} $\aaz_{-5/3}$-module.  Because the small Weyl groups of $\mdl_{\Lam_1}$ and $\mdl_{\Lam_2}$ are both easily checked to be $\ang{\swref{1}}$, we get $6$ distinct $\wgrp$-twists from each, resulting in $12$ simple \infdim{} \hw{} $\aaz_{-5/3}$-modules.  The standard parabolic family found above again has trivial small Weyl group, so it also generates $\abs{\wgrp / \wgrp_{\SLA{sl}{2}}} = 6$ parabolic families under $\wgrp$-twisting.  As mentioned above, there are no coherent families of $\aaz_{-5/3}$-modules.  Again, the corresponding classification of simple relaxed \hw{} $\affvoa{-5/3}{\alg{g}_2}$-modules, with \fdim{} weight spaces, is surely new.

\subsection{Example: $\sovoa{-2}{8}$} \label{ex:d4}

We finish up with a more challenging example, both to illustrate the power of our classification results but also to point out that non-admissible levels have some slightly different features.

The fourth member of the Deligne exceptional series of affine \voas{} corresponds to $\ag=\SLA{so}{8}$ at the non-admissible level $\kk=-2$.  We choose simple roots $\sroot{1},\dots,\sroot{4}$, ordered so that $\sroot{2}$ corresponds to the centre of the Dynkin diagram.  The fundamental weights then have the form
\begin{equation}
	\fwt{2} = \sroot{2} + \sum_{i=1}^4 \sroot{i}, \qquad
	\fwt{i} = \frac{1}{2} (\sroot{i} + \sroot{2}) + \frac{1}{2} \sum_{i=1}^4 \sroot{i}, \quad i=1,3,4.
\end{equation}

The simple \hwms{} over the Zhu algebra $\aaz_{-2}$ were classified in \cite{PerNot13}.  The result is that there is again a unique simple \fdim{} module $\mdl_0$, while now we have four \infdim{} simples: $\mdl_{\Lam_i}$, $i=1,2,3,4$, with $\Lam_2=-\fwt{2}$ and $\Lam_i=-2\fwt{i}$ for $i=1,3,4$.

The subsets $S \subseteq \set{1,2,3,4}$ that lead to \type{AC}-type parabolic subalgebras are easy to find:
\begin{enumerate}
	\item $S=\set{1},\set{3},\set{4}$, giving $\al \cong \SLA{sl}{2} \oplus \SLA{gl}{1}^{\oplus 3}$; \label{it:first}
	\item $S=\set{2}$, giving $\al \cong \SLA{sl}{2} \oplus \SLA{gl}{1}^{\oplus 3}$;
	\item $S=\set{1,2},\set{2,3},\set{2,4}$, giving $\al \cong \SLA{sl}{3} \oplus \SLA{gl}{1}^{\oplus 2}$;
	\item $S=\set{1,3},\set{1,4},\set{3,4}$, giving $\al \cong \SLA{sl}{2}^{\oplus 2} \oplus \SLA{gl}{1}^{\oplus 2}$;
	\item $S=\set{1,3,4}$, giving $\al \cong \SLA{sl}{2}^{\oplus 3} \oplus \SLA{gl}{1}$.
	\item $S=\set{1,2,3},\set{1,2,4},\set{2,3,4}$, giving $\al \cong \SLA{sl}{4} \oplus \SLA{gl}{1}$. \label{it:last}
\end{enumerate}
Note that the subsets listed in each case are all related by the $\ZZ_3$ outer automorphism of $\SLA{so}{8}$, represented on the Dynkin node labels by $1 \to 3 \to 4 \to 1$.  We shall therefore only need to analyse parabolic families from one representative of each case, the others then following by twisting by outer automorphisms.

We now classify the parabolic families of $\aaz_{-2}$-modules for each of the cases \ref{it:first}--\ref{it:last} above.
\begin{enumerate}
	\item When $S=\set{1}$, projecting the $\Lam_i$ onto the $\as \cong \SLA{sl}{2}$ weight space spanned by $\sroot{1}$ results in $\pi_{\as}(\Lam_1) = -2\fwt{1}^{\as}$ and $\pi_{\as}(\Lam_i)=0$ for $i=2,3,4$.  As $-2\fwt{1}^{\as} \in \bounded_{\SLA{sl}{2}}$ is integral, the connected component has a single element and so there is a single associated coherent family $\mdc^1 = \mdc_0 \otimes \CC_{-\fwt{2}}$ of $\al$-modules.  It induces to a standard parabolic family $\mdp^1 = \pind{\ap} \mdc^1$ of $\aaz_{-2}$-modules that contains $\mdl_{\Lam_1}$.  Interestingly, because the coherent family $\mdc_0$ of $\SLA{sl}{2}$-modules contains the trivial $\SLA{sl}{2}$-module $\CC_0$, it follows that $\mdp^1$ contains $\pind{\ap}(\CC_0 \otimes \CC_{-\fwt{2}}) \cong \mdl_{-\fwt{2}} = \mdl_{\Lam_2}$.  Twisting by $\ZZ_3$, we generate two other standard parabolic families of $\aaz_{-2}$-modules which we shall denote by $\mdp^3$ and $\mdp^4$.  The family $\mdp^i$, $i=1,3,4$, therefore corresponds to $S=\set{i}$ and contains both $\mdl_{\Lam_2}$ and $\mdl_{\Lam_i}$.\label{it:1}
	\item When $S=\set{2}$, projecting gives $\pi_{\as}(\Lam_i)=0$, for $i=1,3,4$, and $\pi_{\as}(\Lam_2)=-\fwt{1}^{\as} \in \bounded_{\SLA{sl}{2}}$.  This corresponds to a single coherent family $\mdc^2 = \mdc_{-1/2} \otimes \CC_{\mu}$ of $\al$-modules, where $\mu=-\frac{1}{2}(\fwt{1}+\fwt{3}+\fwt{4})$.  Inducing therefore gives a standard parabolic family $\mdp^2$ of $\aaz_{-2}$-modules that contains $\mdl_{\Lam_2}$.  Note that $\mdc_{-1/2}$ contains only one simple \hw{} $\SLA{sl}{2}$-module, hence $\mdp^2$ contains only one simple \hw{} $\SLA{so}{8}$-module.
	\item When $S=\set{1,2}$, projecting onto the $\as \cong \SLA{sl}{3}$ weight space spanned by $\sroot{1}$ and $\sroot{2}$ gives $\pi_{\as}(\Lam_1) = -2\fwt{1}^{\as}$ and $\pi_{\as}(\Lam_2)=-\fwt{2}^{\as}$, while both $\Lam_3$ and $\Lam_4$ give $0$.  As $\pi_{\as}(\Lam_1)$ and $\pi_{\as}(\Lam_2)$ are both shifted-singular integral weights in $\bounded_{\SLA{sl}{3}}$, the corresponding connected component has two elements.  This thus yields precisely one coherent family $\mdc^{1,2} = \mdc \otimes \CC_{\nu}$ of $\al$-modules, where $\nu = -\frac{2}{3} (\fwt{3}+\fwt{4})$ and $\mdc$ is determined by its $\SLA{sl}{3}$ central character (which coincides with that of the $\SLA{sl}{3}$-modules of highest weights $\pi_{\as}(\Lam_1)$ and $\pi_{\as}(\Lam_2)$).

	\noindent Note that the eigenvalue of the quadratic Casimir of $\mdc$ is
	\begin{equation}
		\bilin{\pi_{\as}(\Lam_1)}{\pi_{\as}(\Lam_1) + 2 \wvec^{\as}} = \bilin{\pi_{\as}(\Lam_2)}{\pi_{\as}(\Lam_2) + 2 \wvec^{\as}} = -\frac{4}{3} < 0.
	\end{equation}
	Because the Casimir eigenvalue on any \fdim{} simple $\SLA{sl}{3}$-module is non-negative, it follows that $\mdc$ contains no such modules.  Its only \hwms{} are thus the \infdim{} ones already found.  We can therefore now conclude that the standard parabolic family $\mdp^{1,2}$ of $\aaz_{-2}$-modules that we obtain by inducing contains only two \hw{} $\SLA{so}{8}$-modules: $\mdl_{\Lam_1}$ and $\mdl_{\Lam_2}$.  There are thus two other inequivalent standard parabolic families $\mdp^{2,3}$ and $\mdp^{2,4}$ which may be obtained as $\ZZ_3$-twists of $\mdp^{1,2}$.
	\item When $S=\set{1,3}$, we have $\as = \as_1 \oplus \as_3$, where $\as_1$ is the $\SLA{sl}{2}$ subalgebra corresponding to $\alpha_1$ and $\as_3$ is the $\SLA{sl}{2}$ subalgebra corresponding to $\alpha_3$.  From case \ref{it:1}, the only $\Lam_j$ whose $\as_i$-projection lands in $\bounded_{\as_i}$, for $i=1$ or $3$, is $\Lam_i$.  As no $\Lam_j$ satisfies this boundedness criterion for both $i=1$ and $3$, there is no coherent family of $\al$-modules, hence no parabolic family of $\aaz_{-2}$-modules.  The same is obviously true for $S=\set{1,4}$ and $\set{3,4}$.
	\item When $S=\set{1,3,4}$, there are likewise no parabolic families of $\aaz_{-2}$-modules for the same reason as in the previous case.
	\item Finally, when $S=\set{1,2,3}$, we have $\as \cong \SLA{sl}{4}$ and the projections are $\pi_{\as}(\Lam_1) = -2\fwt{1}^{\as}$, $\pi_{\as}(\Lam_2) = -\fwt{2}^{\as}$, $\pi_{\as}(\Lam_3) = -2\fwt{3}^{\as}$ and $\pi_{\as}(\Lam_4) = 0$.  All but the last belong to $\bounded_{\SLA{sl}{4}}$ and may be checked to be shifted-singular and integral.  They therefore form a single connected component, hence we get one coherent family $\mdc^{1,2,3} = \mdc \otimes \CC_{-\fwt{4}}$ of $\al$-modules.  Here, $\mdc$ is determined by its central character which agrees with that of the $\SLA{sl}{4}$-modules of highest weights $\pi_{\as}(\Lam_i)$, $i=1,2,3$.  In particular, the common quadratic Casimir eigenvalue is $-3$ which again rules out $\mdc$ containing any \fdim{} $\SLA{sl}{4}$-modules.  In this way, we arrive at one standard parabolic family $\mdp^{1,2,3}$ of $\aaz_{-2}$-modules that contains only the $\mdl_{\Lam_i}$ with $i=1,2,3$.  $\ZZ_3$-twisting gives two more standard parabolic families which we shall denote by $\mdp^{1,2,4}$ and $\mdp^{2,3,4}$.
\end{enumerate}

This gives us all the standard parabolic families of $\aaz_{-2}$-modules.  It only remains to determine how many non-standard families there are.  For this, we recall that the Weyl group $\wgrp$ of $\SLA{so}{8}$ is isomorphic to $\grp{S}_4 \ltimes \ZZ_2^3$ and so has order $4! \cdot 2^3 = 192$.  We tabulate the small Weyl groups of each family as well as the Weyl groups of the corresponding Levi factors $\al$:
\begin{center}
	\begin{tabular}{C|CCCCCCC}
		\mdm & \mdl_0 & \mdl_{\Lam_1},\ \mdl_{\Lam_3},\ \mdl_{\Lam_4} & \mdl_{\Lam_2} & \mdp^1,\ \mdp^3,\ \mdp^4 & \mdp^2 & \mdp^{1,2},\ \mdp^{2,3},\ \mdp^{2,4} & \mdp^{1,2,3},\ \mdp^{1,2,4},\ \mdp^{2,3,4} \\
		\hline
		\swgrp{\mdm} & \wgrp & \grp{S}_4 & \ZZ_2^3 & \ZZ_2^2 & \wun & \wun & \wun \\
		\wgrp_{\al} & \wun & \wun & \wun & \ZZ_2 & \ZZ_2 & \grp{S}_3 & \grp{S}_4 \\
		\dfrac{\abs{\wgrp}}{\abs{\swgrp{\mdm}} \abs{\wgrp_{\al}}} & 1 & 8 & 24 & 24 & 96 & 32 & 8
	\end{tabular}
\end{center}
We therefore have $1$ \fdim{} simple module, $3 \times 8 + 24 = 48$ \infdim{} \hwms{}, $3 \times 24 + 96 = 168$ one-parameter parabolic families with $\al \cong \SLA{sl}{2}$, $3 \times 32 = 96$ two-parameter parabolic families with $\al \cong \SLA{sl}{3}$, and $3 \times 8 = 24$ three-parameter parabolic families with $\al \cong \SLA{sl}{4}$.  This gives the complete classification of simple weight $\aaz_{-2}$-modules with \fdim{} weight spaces.  We are quietly confident that the corresponding classification of simple relaxed \hw{} $\sovoa{-2}{8}$-modules (with \fdim{} weight spaces) was unknown before now.

Note finally the unexpected, and therefore interesting, fact that the parabolic families $\mdp^1$, $\mdp^3$ and $\mdp^4$ all contain $\mdl_{\Lambda_2}$.  This is due to the fact that each of the corresponding coherent families of $\SLA{sl}{2}$-modules has a finite-dimensional \hw{} submodule.  We did not observe integral highest weights upon projecting the admissible weights of the previous examples, so it is reasonable to conjecture that this is a feature of non-admissible levels.  In this example, the projected weights were always shifted-singular, so the coherent families are completely characterised by their central characters.  It would be very interesting, and perhaps a little alarming, to find an example with shifted-regular projections, hence coherent families that cannot be distinguished by their central characters alone.

\section{An application to category $\categ{O}$} \label{sec:catO}

The previous \lcnamecref{sec:ex} detailed many examples of applications of \cref{thm:zhuclass}, hence \cref{thm:classification}.  In this \lcnamecref{sec:catO}, we outline an application of \cref{thm:zhuind}, hence \cref{thm:parabolicindecomposable}.

For an admissible-level affine \voa{} (like those studied in \cref{ex:a1,ex:a2,ex:c2,ex:g2}), the modules in category $\categ{O}$ are known to be semisimple.  This was originally conjectured in \cite{AdaVer95} and proven in \cite{AraRat16}.  Here, we pose the question of whether a \emph{quasilisse} affine \voa{} can have a non-semisimple module in category $\categ{O}$.  Recall that quasilisse \voas{} are generalisations, introduced in \cite{AraQua16}, of the well-known lisse, or $C_2$-cofinite, \voas{}.  Admissible-level affine \voas{} are always quasilisse \cite{AraQua16}, but there are also quasilisse affine examples with non-admissible levels.  We shall answer the question posed above in the affirmative by establishing that $\categ{O}$ is non-semisimple for the quasilisse, but non-admissible-level, affine \voa{} $\sovoa{-2}{8}$, studied in \cref{ex:d4}.

To establish this, we shall construct a non-semisimple extension $\mdm$ of two \hw{} $\zhu{\sovoa{-2}{8}}$-modules as a submodule of a non-semisimple parabolic family $\mdp$ of $\SLA{so}{8}$-modules.  We will then apply \cref{thm:zhuind} to show that $\ind{\mdp}$, and therefore the submodule $\ind{\mdm}$, is an $\sovoa{-2}{8}$-module (see \cref{prop:d4ses} below for the precise result).  We use the same notation as in \cref{ex:d4}.

Recall that there are five simple \hw{} $\sovoa{-2}{8}$-modules, up to isomorphism, and that their highest weights are $\Lambda_0 = -2 \fwt{0}$ (the vacuum module), $\Lambda_2 = -\fwt{2}$ and $\Lambda_i = -2\fwt{i}$, $i=1,3,4$.  The conformal weights of their \hwvs{} are easily established to be $0$, for the vacuum, and $-1$ otherwise.  Let $\ap$ be the standard parabolic subalgebra defined by the subset $\set{1}$ of simple root labels.  The Levi factor of $\ap$ is then $\al \cong \as \oplus \SLA{gl}{1}^{\oplus 3}$ with $\as\cong \SLA{sl}{2}$.

Our first task is to construct the irreducible non-semisimple parabolic family $\mdp$ of $\SLA{so}{8}$-modules.  This parabolic family will be $-\alpha_1$-bijective (see \cref{def:bijective}) and shall contain the simple \hw{} $\SLA{so}{8}$-module $\mdl_{\Lambda_1}$ as a submodule. This will allow us to apply \cref{thm:zhuind} to conclude that $\ind{\mdp}$ is an $\sovoa{-2}{8}$-module.  The key step in this construction is the existence of an irreducible $-\sroot{}$-bijective coherent family $\mdc$ of $\as$-modules ($\sroot{}$ denoting the simple root of $\as \cong \SLA{sl}{2}$) that contains the bounded \hw{} $\as$-module $\mdl_{-\sroot{}}^{\as}$ as a submodule.   Tensoring with an appropriate $\SLA{gl}{1}^{\oplus 3}$-module and inducing with $\pind{\ap}$ will then give the desired parabolic family $\mdp$.  We shall outline a construction of the coherent family $\mdc$ using induction for completeness.  Readers for whom the existence of $\mdc$ is clear may safely skip the next two paragraphs.

Recall that one may construct a dense $\as$-module $\mdd_{\lambda;q}$ by inducing the one-dimensional module $\CC_{\lambda;q}$ of the centraliser $\cent{\ahs}{\as} \cong \CC[h,\cas]$:
\begin{equation}
	\mdd_{\lambda;q}=\envalg{\as}\otimes_{\cent{\ahs}{\as}}\CC_{\lambda;q}.
\end{equation}
Here, $\ahs = \vspn \set{h}$ is the Cartan subalgebra of $\as$, $\cas$ is the quadratic Casimir of $\as$, $\lambda$ is the unique weight of $\CC_{\lambda;q}$, and $q$ is the $\cas$-eigenvalue.  It is easy to show that $\mdd_{\lambda;q}$ has 1-dimensional weight spaces with weight support $\lambda+\rlat_{\as}$ (further details may be found, for example, in \cite[Ex.~3.99]{MazLec10} or \cite[Sec.~3.2]{KawRel18}).  We set $q$ to $0$, noting that it follows that the weight of any \hw{} or \lwv{} in $\mdd_{\lambda;0}$ belongs to $\set{0, \pm \sroot{}}$.  We conclude that the $\mdd_{\lambda;0}$ are simple and dense, hence $-\sroot{}$-bijective, whenever $\lambda$ does not lie in $\rlat_{\as}$.

Now, $\mdd_{\lambda;0}$ is not simple for $\lambda \in \rlat_{\as}$.  In particular, setting $\lambda$ to $\sroot{}$ results in a dense module $\mdd_{\sroot{};0}$ containing \hwvs{} $u$ and $v$ of $\as$-weights $-\sroot{}$ and $0$, respectively,
satisfying
\begin{equation} \label{eq:soeasy}
 e^{\sroot{}} u = 0, \qquad e^{\sroot{}}v=0  \qquad \text{and} \qquad u = e^{-\sroot{}} v.
\end{equation}
Here, $e^{\pm\sroot{}}$ denotes the root vector of $\as$ corresponding to the root $\pm\sroot{}$.  The composition factors of $\mdd_{\sroot{};0}$ are thus the simple \hw{} $\as$-modules $\mdl_{-\sroot{}}^{\as}$ and $\mdl_0^{\as}$, along with the simple \lw{} $\as$-module $\wref{}(\mdl_{-\sroot{}}^{\as})$, where $\wref{}$ denotes the Weyl reflection of $\as$.  We illustrate the structure of $\mdd_{\sroot{};0}$ in \cref{fig:sl2coh}.  What is most important here, however, is that $\mdd_{\sroot{};0}$ is $-\sroot{}$-bijective.  The direct sum
\begin{equation} \label{eq:socoherent}
	\mdc = \bigoplus_{0 < t \le 1} \mdd_{t \sroot{};0}
\end{equation}
is therefore the desired irreducible $-\sroot{}$-bijective coherent family of $\as \cong \SLA{sl}{2}$-modules.  Note that the \hw{} submodule $\mdl_{-\sroot{}}^{\as} \subset \mdc$ is clearly bounded.

\begin{figure}
	\begin{tikzpicture}[>=stealth,very thick,scale=1.75]
		\node (4) at (3.75,0) {$\cdots$};
		\node[bwt,label=below:{$3\sroot{}$}] (3) at (3,0) {};
		\node[bwt,label=below:{$2\sroot{}$}] (2) at (2,0) {};
		\node[bwt,label=below:{$\sroot{}$}] (1) at (1,0) {};
		\node[bwt,label=below:{$0$},label=above:{$v$}] (0) at (0,0) {};
		\node[bwt,label=below:{$-\sroot{}$},label=above:{$u$}] (-1) at (-1,0) {};
		\node[bwt,label=below:{$-2\sroot{}$}] (-2) at (-2,0) {};
		\node[bwt,label=below:{$-3\sroot{}$}] (-3) at (-3,0) {};
		\node (-4) at (-3.75,0) {$\cdots$};
		\draw[<->] (-4) -- (-3);
		\draw[<->] (-3) -- (-2);
		\draw[<->] (-2) -- (-1);
		\draw[<-] (-1) -- (0);
		\draw[<-] (0) -- (1);
		\draw[<->] (1) -- (2);
		\draw[<->] (2) -- (3);
		\draw[<->] (3) -- (4);
	\end{tikzpicture}
\caption{An illustration of the structure of the non-semisimple dense $\as \cong \SLA{sl}{2}$-module $\mdd_{\sroot{};0}$.  Weights are indicated with black circles and the corresponding weight spaces are all $1$-dimensional.  The \hwvs{} $u$ and $v$ are also indicated above the circles representing their weights.  Arrows pointing right (left) are drawn whenever the action of the $\SLA{sl}{2}$ root vector $e=e^{\sroot{}}$ ($f=e^{-\sroot{}}$) is bijective.  Note that the vector of weight $\sroot{}$ indeed generates the entire module.} \label{fig:sl2coh}
\end{figure}

Having constructed $\mdc$, we now lift it to a coherent family of modules over $\al \cong \SLA{sl}{2} \oplus \SLA{gl}{1}^{\oplus 3} \subseteq \SLA{so}{8}$ by identifying the $\SLA{sl}{2}$ simple root $\sroot{}$ with its $\SLA{so}{8}$ counterpart $\sroot{1}$ and tensoring with the one-dimensional $\SLA{gl}{1}^{\oplus 3}$-module $\CC_{-\fwt{2}} = \vspn \set{w}$ whose sole weight is $-\fwt{2}$.  It follows that $\mdc \otimes \CC_{-\fwt{2}}$ has \hwvs{} $u \otimes w$ and $v \otimes w$ of weights $-\sroot{1} - \fwt{2} = -2 \fwt{1} = \Lambda_1$ and $-\fwt{2} = \Lambda_2$.  Note that $u \otimes w$ generates a simple bounded \hw{} $\al$-submodule of $\mdc \otimes \CC_{-\fwt{2}}$.

Form the parabolic family $\mdp = \pind{\ap}(\mdc \otimes \CC_{-\fwt{2}})$ of $\SLA{so}{8}$-modules.  Because $\ap$ contains the standard Borel, the image of any \hwv{} under $\mdc \otimes \CC_{-\fwt{2}} \ira \mdp$ will again be a \hwv{}.  Moreover, the image of $u \otimes w$ generates a \emph{simple} \hw{} $\SLA{so}{8}$-submodule of $\mdp$, by \cref{prop:functors}\ref{it:functrex} and \ref{it:functsimp}.  This submodule is obviously $\al$-bounded and isomorphic to $\mdl_{\Lambda_1}$.  The image $\ind{\mdl_{\Lambda_1}}$ under the Zhu-induction functor is therefore the simple \hw{} $\sovoa{-2}{8}$-module of highest weight $\Lambda_1$.  By \cref{thm:zhuind}, we may now conclude that every subquotient of $\ind{\mdp}$ is also an $\sovoa{-2}{8}$-module.

In particular, the Zhu-induction of the \hw{} submodule $\mdm$ of $\mdp$ generated by (the image of) the \hwv{} $v \otimes w$ is an $\sovoa{-2}{8}$-module.  It follows from \cref{thm:zhu}\ref{it:zhuindec} that $\mdm$ has $\mdl_{\Lambda_2}$ as a composition factor (because $v \otimes w$ has weight $\Lambda_2$).  However, $u \otimes w$ is also a \hwv{} in $\mdm$, by \eqref{eq:soeasy}, hence $\mdm$ has $\mdl_{\Lambda_1}$ as another composition factor.  As $\mdm$ is \hw{}, it is indecomposable and so we have proved the following \lcnamecref{prop:d4indec}.  We only need remark that the outer automorphisms of $\SLA{so}{8}$ allow us to swap $\Lambda_1$ for $\Lambda_3$ or $\Lambda_4$ in this conclusion.
\begin{proposition} \label{prop:d4indec}
	There exist non-simple \hw{} $\sovoa{-2}{8}$-modules.  In particular, for each $i=1,3,4$, there exists a \hw{} $\sovoa{-2}{8}$-module whose composition factors include $\ind{\mdl_{\Lambda_2}}$ and $\ind{\mdl_{\Lambda_i}}$.
\end{proposition}
\noindent To the best of our knowledge, this is the first time that non-semisimplicity has been demonstrated in category $\categ{O}$ for a quasilisse affine \voa{} (non-quasilisse examples with $\categ{O}$ non-semisimple are already known, see \cite[Rem.~5.8]{AdaRep08} and \cite[Thm.~7.2]{AraShe17}).  In this regard, it is also interesting to note that the character of the vacuum $\sovoa{-2}{8}$-module $\ind{\mdl_0}$ is quasimodular, but not modular \cite{AraQua16}.  See also \cite{KawWal18} for a relation between non-admissible levels and semisimplicity.

It is in fact easy to show that the non-simple \hw{} $\sovoa{-2}{8}$-modules that we have constructed have no other composition factors.  First, consideration of the conformal weights of the \hwvs{} of the five possible isomorphism classes of composition factors lets us conclude that any composition factor of $\ind{\mdm}$ with non-zero highest weight will already be detected as a composition factor of $\mdm \subset \mdp$.  Clearly, $\mdl_{\Lambda_2}$ appears with multiplicity $1$ in $\mdm$.  Because $\Lambda_i = \Lambda_2 - \sroot{i}$, for $i=1,3,4$, the multiplicity of $\mdl_{\Lambda_i}$ in $\mdm$ is at most $1$.  \cref{prop:d4indec} then shows that this multiplicity is exactly $1$ for $i=1$.

To determine the multiplicity of $\mdl_{\Lambda_j}$, $j=3,4$, it suffices to check the subsingularity of the vector $e^{-\sroot{j}} (v \otimes w)$ in $\mdp$.  However, acting with any simple root vector necessarily gives $0$, so we conclude that this vector is either singular or zero.  Either way, the $\SLA{so}{8}$-module it generates in $\mdp$ has zero intersection with $\uinv{\ap} \mdp = \mdc \otimes \CC_{-\fwt{2}}$, hence $e^{-\sroot{j}} v$ must be zero in $\mdp$ by definition.  As it must therefore also be zero in $\mdm$, we conclude that $\mdl_{\Lambda_j}$ is not a composition factor of $\mdm$, hence that $\ind{\mdl_{\Lambda_j}}$ is not a composition factor of $\ind{\mdm}$, for $j=3,4$.

A similar argument rules out the vacuum module $\ind{\mdl_{\Lambda_0}}$ appearing as a composition factor.  For this, we note that the highest root of $\SLA{so}{8}$ is $\hroot = \fwt{2}$.  It follows that the multiplicity of $\ind{\mdl_{\Lambda_0}}$ in $\ind{\mdm}$ is also at most $1$.  This time, the precise multiplicity may be determined by studying the subsingularity of $e^{\hroot}_{-1} v \in \ind{\mdm}$.  However, this is easily shown to be a \sv{} or zero, using $\kk=-2$.  Either way, the submodule it generates has zero intersection with $\mdm$, hence it vanishes by \cref{thm:zhu}\ref{it:zhuindec}.
This gives the desired conclusion.
\begin{proposition} \label{prop:d4ses}
	For $\sovoa{-2}{8}$, there exist in category $\categ{O}$ \emph{non-split} extensions of $\ind{\mdl_{\Lambda_2}}$ by $\ind{\mdl_{\Lambda_i}}$, where $i=1,3,4$.  In other words, there exist indecomposable $\sovoa{-2}{8}$-modules $\ind{\mdm_i}$, $i=1,3,4$, in $\categ{O}$ such that
	\begin{equation}
		\dses{\ind{\mdl_{\Lambda_i}}}{}{\ind{\mdm_i}}{}{\ind{\mdl_{\Lambda_2}}}
	\end{equation}
	is exact.
\end{proposition}
\noindent As before, this follows for $i=3,4$ from the argument above for $\mdm_1 = \mdm$ by applying outer automorphisms.

We conclude by noting that the above arguments also establish the existence of non-split short exact sequences of modules over $\envalg{\SLA{so}{8}} / \aaj$, where $\aaj$ is the Joseph ideal of $\SLA{so}{8}$.  Recall that when $\ag$ is not of type \type{A}, there is a unique completely prime primitive ideal of $\envalg{\ag}$, the \emph{Joseph ideal}, whose associated variety is the closure of the minimal nilpotent orbit in $\ag^*$.  As $\zhu{\sovoa{-2}{8}} \cong \CC \times \envalg{\SLA{so}{8}} / \aaj$ \cite[Thm.~3.1]{AraJos18}, it follows that a $\zhu{\sovoa{-2}{8}}$-module with no composition factor isomorphic to $\mdl_0$ is automatically a $\envalg{\SLA{so}{8}} / \aaj$-module.
\begin{corollary} \label{cor:joseph}
	There exist non-split short exact sequences of $\envalg{\SLA{so}{8}} / \aaj$-modules, including
	\begin{equation}
		\dses{\mdl_{\Lambda_i}}{}{\mdm_i}{}{\mdl_{\Lambda_2}}, \qquad i=1,3,4.
	\end{equation}
\end{corollary}
\noindent We thank Tomoyuki Arakawa for pointing out this interesting consequence of our work.

\flushleft
\providecommand{\opp}[2]{\textsf{arXiv:\mbox{#2}/#1}}
\providecommand{\pp}[2]{\textsf{arXiv:#1 [\mbox{#2}]}}

\end{document}